\documentclass{article}
\usepackage[scale=0.7]{geometry}
\usepackage{amsmath}
\usepackage{amssymb}
\usepackage{amsfonts}
\usepackage{amsthm}
\usepackage{enumitem}
\usepackage{mathrsfs}
\usepackage{hyperref}
\hypersetup{colorlinks = true, linkcolor = blue, anchorcolor = blue, citecolor =blue}
\usepackage{cleveref}
\crefname{equation}{}{}
\usepackage{algorithm}
\usepackage{algorithmic}
\usepackage{graphicx}
\usepackage{float} 
\usepackage{authblk}

\newtheorem{theorem}{Theorem}
\newtheorem{lemma}{Lemma}
\newtheorem{corollary}{Corollary}

\newtheorem{definition}{Definition}
\newtheorem{assumption}{Assumption}
\newtheorem{remark}{Remark}

\newcommand{\x}{\boldsymbol{x}}

\newcommand{\y}{\boldsymbol{y}}
\newcommand{\R}{\mathbb{R}}
\newcommand{\bm}{\boldsymbol}
\newcommand{\abs}[1]{\left| #1 \right|}
\title{Nested Multilevel Monte Carlo with Preintegration for Efficient Risk Estimation}

\author[1]{Yu Xu \thanks{Corresponding author: yu-xu22@mails.tsinghua.edu.cn}}
\author[2]{Xiaoqun Wang}

\affil[1, 2]{Department of Mathematical Sciences, Tsinghua University, Beijing, China}

\date{\today}

\begin{document}

\maketitle

\begin{abstract}
Nested Monte Carlo is widely used for risk estimation, but its efficiency is limited by the discontinuity of the indicator function and high computational cost. This paper proposes a nested Multilevel Monte Carlo (MLMC) method combined with preintegration for efficient risk estimation. We first use preintegration to integrate out one outer random variable, which effectively handles the discontinuity of the indicator function, then we construct the MLMC estimator with preintegration to reduce the computational cost. Our theoretical analysis proves that the strong convergence rate of the MLMC combined with preintegration reaches -1, compared with -1/2 for the standard MLMC. Consequently, we obtain a nearly optimal computational complexity. Besides, our method can also handle the high-kurtosis phenomenon caused by indicator functions. Numerical experiments verify that the smoothed MLMC with preintegration outperforms the standard MLMC and the optimal computational cost can be attained. Combining our method with quasi-Monte Carlo further improves its performance in high dimensions.

\end{abstract}

\section{Introduction}
Risk estimation is a basic and important task in financial risk management. It is used in many real financial activities, such as portfolio investment, option pricing and bank risk control. The goal of risk estimation is to calculate the probability of extreme loss events for a financial portfolio accurately and quickly. In most practical cases, this calculation can be transformed into estimating the expectation of a function of a conditional expectation. For this kind of problem, nested Monte Carlo simulation \cite{ broadie2011_nested, nested_bias2010} is the most commonly used method, because it is easy to implement and can adapt to complex financial models.

Nested Monte Carlo simulation works with two layers of sampling: the outer layer generates samples of the multidimensional random variable, and the inner layer generates samples to estimate the conditional expectation for each outer sample. After getting the conditional expectation estimate for all outer samples, we aggregate these results to get the final estimate of the target risk measure.  Gordy and Juneja \cite{nested_bias2010} showed that to achieve a root mean squared error (RMSE) for given  tolerance $\text{TOL}$, allocating proper computational effort in each level results in a total  cost of $O(\text{TOL}^{-3})$. Nested simulation can be made more efficient by allocating computational effort nonuniformly across scenarios. Particularly, Broadie et al. \cite{broadie2011_nested} proposed an adaptive procedure to allocate computational effort to inner simulations. The resulting nonuniform nested simulation estimator enjoys a reduced complexity of $O(\text{TOL}^{-5/2})$ under certain conditions.

To reduce the computational cost of nested Monte Carlo simulation, Multilevel Monte Carlo (MLMC) \cite{giles_2008_mlmc, giles_2015_mlmc, heinrich1998} has been developed and widely used in recent years. MLMC is a powerful variance reduction technique. It decomposes the target expectation into a sum of differences between estimates at different accuracy levels. For high-accuracy levels with high sampling cost, MLMC uses fewer samples, and for low-accuracy levels with low cost, it uses more samples. This optimal sample allocation greatly reduces the total computational cost compared with the standard Monte Carlo method. Recently, Giles and Haji-Ali \cite{giles_2019_mlmc} proposed to use MLMC in nested simulation for risk estimation. They showed that in the original MLMC, the complexity is $O(\text{TOL}^{-5/2})$. By incorporating the adaptive allocations procedure of \cite{broadie2011_nested}, Giles and Haji-Ali \cite{giles_2019_mlmc} showed that the complexity of MLMC can be reduced to $O(\text{TOL}^{-2}(\log (\text{TOL}))^2)$ under certain conditions. For some applications of nested MLMC, we refer to \cite{bujok2015multilevel,giles2018mlmc,goda_nested_max,goda_nested_log,goda2018decision}.

However, the standard MLMC still cannot solve the problem of the indicator function's discontinuity in risk estimation. This discontinuity makes the standard MLMC have poor robustness: its variance decays very slowly at deep high-accuracy levels, and the kurtosis increases sharply with the level. As a result, the standard MLMC still has limited efficiency in practical risk estimation tasks.
To address the discontinuity problem of the indicator function, preintegration \cite{achtsis2013conditionalB, achtsis2013conditionalA, bayer2018smoothing, glasserman2001conditioning, griewank2018high} (is also called conditional MC or conditional sampling) techniques have become a research focus in recent years. Preintegration is one of these effective smoothing techniques. It integrates out one of the random variables analytically or numerically, which can eliminate the discontinuity caused by the indicator function. What's more, preintegration is easy to combine with the existing MLMC framework, and it does not change the original bias of the nested simulation estimator. This means that preintegration can improve the simulation's convergence and robustness without increasing the bias, which makes it a perfect match for MLMC in risk estimation.

In this paper, we combine the preintegration technique with nested MLMC to develop a new efficient method for risk estimation. We start with the basic problem setting of nested simulation for risk estimation, and review the asymptotic bias property of the standard nested Monte Carlo estimator from existing research \cite{nested_bias2010}. We then introduce the preintegration technique in detail: we define the preintegrated function for the risk measure. When a closed-form explicit expression for the function after preintegration does not exist, Bayer et al. \cite{bayer2023numerical,bayer2024multilevel} proposed a numerical smoothing method to approximate the preintegration via Newton iteration and one-dimensional numerical quadrature schemes. Bayer et al. use the Newton iteration method to approximate the discontinuity point, and apply the two-sided Laguerre quadrature rule to compute the one-dimensional preintegration numerically.

After building the smoothed estimator with preintegration, we combine it with the MLMC method to construct the nested MLMC estimator with preintegration. Then we conduct a comprehensive and detailed theoretical analysis for the proposed method. First, we decompose the mean squared error (MSE) of the proposed estimator into three parts: the bias error of the nested simulation, the statistical error of the MLMC method, and the numerical error from preintegration. We analyze the convergence rate of each error part separately. Second, we prove the strong convergence of the proposed method, and derive that the variance of the level differences decays at a rate of $O(m_\ell^{-1})$, where $m_\ell$ is the inner sample size at level $\ell$. Third, we analyze the computational complexity of the method, and show that it has a complexity of $O(\text{TOL}^{-2}(\log(\text{TOL}))^2)$. Fourth, we prove that the kurtosis of the proposed method is bounded (i.e., $O(1)$), which means the method has good robustness and avoids the high-kurtosis problem of the standard MLMC.

The numerical experiment results confirm our theoretical conclusions. The smoothed MLMC with preintegration has more stable weak convergence than the standard MLMC, and its variance decays fast and nearly achieves our theoretical convergence rate. The kurtosis of the proposed method remains at a low level and almost does not change with the level, which means it has excellent robustness. In terms of computational cost, the proposed method needs much less cost to achieve the same accuracy, and its cost attained theoretical complexity $O(\text{TOL}^{-2}(\log(\text{TOL}))^2)$. 
 We also extend the proposed method by combining it with quasi-Monte Carlo \cite{caflisch1998monte, caflisch1997valuation, dick2010digital, niederreiter1992, owen2023practical} (QMC) to solve the accuracy decline problem of inner simulation in high-dimensional cases, and test the performance of  Multilevel QMC \cite{goda2018decision, xzh} (MLQMC) combined with preintegration (smoothed MLQMC) in high dimensions.
 Smoothed MLQMC significantly improves the computational accuracy in high-dimensional settings, and its strong convergence rate even reaches $-3/2$ in our experiments. Furthermore, combined with antithetic sampling \cite{bujok2015multilevel, goda_nested_max, giles_2019_mlmc}, the above methods achieve further variance reduction and even higher strong convergence rates, thereby effectively reducing the computational complexity.

This paper makes three main contributions to the research of risk estimation with Monte Carlo methods. First, we successfully combine the preintegration smoothing technique with nested MLMC, and construct a new smoothed MLMC estimator for risk estimation. Second, we conduct a complete theoretical analysis for the proposed method, including error decomposition, strong convergence, complexity and robustness, and derive clear theoretical results about convergence rate, variance decay and complexity. Third, we design detailed numerical experiments to verify the theoretical results, and show that the proposed method has obvious advantages over the standard MLMC in practical risk estimation, especially in high-dimensional cases.

The rest of the paper is organized as follows. Section \ref{sec:Nested simulation and preintegration} introduces the basic setting of nested simulation for risk estimation and the preintegration technique. Section \ref{sec:MLMC combined with preintegration} combines preintegration with MLMC, constructs the proposed estimator, and does the detailed theoretical analysis (error, convergence, complexity and robustness). Section \ref{sec:Numerical experiments} presents the numerical experiments and the corresponding results analysis. Section \ref{sec:conclusion} gives the main conclusions of the paper and points out the directions for future research. The Appendix \ref{sec:Appendix} provides the detailed proofs of the lemmas used in the theoretical analysis.

\section{Nested simulation and preintegration}
\label{sec:Nested simulation and preintegration}

\subsection{Problem setting and nested simulation}

In risk estimation, nested Monte Carlo methods are widely adopted to solve the problem of estimating expectations of functions of conditional expectations. We consider the problem of estimating
\begin{equation}
\theta := \mathbb{P} (\mathbb{E}[X|\bm \omega]>c)=\mathbb{E}[\mathbb{I} \{\mathbb{E}[X|\bm \omega]>c\}]:=\mathbb{E}[g(\varphi(\bm \omega))]\label{eq:theta}
\end{equation}
via simulation for a given constant $c$,
where $\varphi (\bm \omega):=\mathbb{E}[X|\bm \omega]$ and $g(\cdot) = \mathbb{I}\{\cdot>c\}$. The inner expectation of the one-dimensional random variable $X$, is conditional on the value of the outer multidimensional random variable $\bm \omega$. 
We aim to estimate the expectation of a function of a conditional expectation via nested simulation, focusing on the problem given in \cref{eq:theta}. In the nested Monte Carlo approach, we consider the estimator
\begin{equation*}
\hat \theta_{n,m} := \frac 1 n \sum_{i=1}^n \hat g_m(\bm \omega_i) ,
\end{equation*}
where
\begin{equation}\label{eq:hat_g_varphi}
\hat g_m(\y) = \mathbb{I} \{\hat \varphi_m(\y) > c\}, \hat \varphi_m(\y) = \frac 1 m\sum_{j=1}^m X_j(\y),
\end{equation}
where $\bm \omega_i$ are independent and identically distributed (i.i.d.) replications of $\bm \omega$, and $X_j(\y)$ are i.i.d. replications of $X$ given $\bm \omega=\y$. The quadrature rule $\hat \varphi_m(\y)$ is used to estimate the conditional expectation $\varphi(\y)=\mathbb{E}[X|\bm \omega=\y]$ in the inner simulation.

The bias of the nested Monte Carlo estimator is a critical factor affecting estimation accuracy, and the following theorem characterizes its asymptotic bias behavior, which is from Proposition 1 in \cite{nested_bias2010}.
\begin{theorem}\label{thm:bias}
	 Let $f(\cdot)$ be the probability density function of $\varphi(\bm \omega)$. Assume the following:
\begin{itemize}
	\item The joint density $p_m(x,y)$ of $\varphi(\bm\omega)$ and $m^{1/2}[\hat{\varphi}_m(\bm\omega)-\varphi(\bm\omega)]$ and partial derivatives $(\partial /\partial x)p_m(x,y)$ and $(\partial^2 /\partial x^2)p_m(x,y)$ exist for each $m$ and $(x,y)$.
	\item For each $m\ge 1$, there exist functions $f_{i,m}(\cdot)$ such that
\begin{equation*}
\bigg|\frac{\partial^i }{\partial x^i}p_m(x,y)\bigg|\le f_{i,m}(y), i=0,1,2.
\end{equation*}
	In addition,
\begin{equation*}	
\sup_{m\ge 1} \int \abs{y}^r f_{i,m}(y) d y<\infty,\ for\ i=0,1,2\ and\ 0\le r\le 4.
\end{equation*}
\end{itemize}
Then the bias of the nested estimator $\hat \theta_{n,m}$ asymptotically satisfies
\begin{equation}\label{eq:bias}
|\mathbb{E}[\hat \theta_{n,m}-\theta] |=|\mathbb{E}[\hat g_m(\bm \omega)-\theta]|\le  \frac{|\Theta'(c)|}{m} + O(m^{-3/2}),
\end{equation}
where
\begin{equation*}
\Theta(c)=\frac 1 2 f(c)\mathbb{E}[\sigma^2(\bm\omega)|\varphi(\bm\omega)=c],
\end{equation*}
and $\sigma^2(\y)$ denotes the conditional variance of $X$ (conditioned on $\bm\omega = \y$).
\end{theorem}
This theorem provides an explicit asymptotic upper bound for the bias of the nested Monte Carlo estimator $\hat \theta_{n,m}$. It demonstrates that the leading term of the bias decays at the rate of $1/m$. 

\subsection{Preintegration}
The efficiency of nested Multilevel Monte Carlo (MLMC) simulation is often significantly degraded by the discontinuity of the indicator function $g(\cdot)$, which slows down the strong convergence rate in practical risk estimation. To avoid this issue and improve the accuracy of nested MLMC simulation, we introduce the preintegration technique, which integrates out one of the outer random variables analytically or numerically, and then the discontinuity can be effectively handled and the overall efficiency can be greatly enhanced.In this paper, for simplicity,  we assume that the variable $\omega_1$ has an independent probability density function $\rho_1$, so we always perform preintegration over $ \omega_1$. The specific procedure of inner simulation combined with preintegration is performed
by the following formula
\begin{align*}
\theta  &= \mathbb{E}[\mathbb{I} \{\varphi (\bm \omega)>c\}] =\mathbb{E}[g(\varphi(\bm \omega))] \\
&= \mathbb{E}[\int_{\mathbb{R}} g(\varphi(\omega_1, \bm \omega_{-1}))\rho_1(\omega_1) d\omega_1] := \mathbb{E}[h(\bm \omega_{-1})] \\
&\approx \mathbb{E}[\int_{\mathbb{R}} \hat g_m(\omega_1, \bm \omega_{-1})\rho_1(\omega_1) d\omega_1] := \mathbb{E}[\hat h_m(\bm \omega_{-1})] ,
\end{align*}
where
\begin{equation*}
    h(\bm \omega_{-1})=\int_{\mathbb{R}} g(\varphi(\omega_1, \bm \omega_{-1}))\rho_1(\omega_1) d\omega_1
\end{equation*}
and
\begin{equation}\label{eq:hat_h_m}
    \hat h_m(\bm \omega_{-1})=\int_{\mathbb{R}} \hat g_m(\omega_1, \bm \omega_{-1})\rho_1(\omega_1) d\omega_1.
\end{equation}

From the above calculations, we can observe that the preintegration does not change the bias of estimator of the nested simulation, and therefore preintegration cannot change the computational complexity of the nested simulation.
It is worth noting that the preintegrated function $\hat h$ generally lacks a closed-form analytical expression, thus requiring numerical approximation of the one-dimensional integral. We follow the numerical smooth method of Bayer et al. \cite{bayer2023numerical, bayer2024multilevel}.
For  the ease of presentation, we make the same assumption as \cite{bayer2024multilevel}, for fixed $\x_{-1}$,  the function $\varphi(x_1,\x_{-1})$  has a simple root or is positive for all $x_1 \in \mathbb{R}$. This  assumption is guaranteed by the following monotonicity condition \ref{condition:a} and    infinite growth condition \ref{condition:b}, which we present the monotonicity condition for an increasing function without  loss of generality.
	
\begin{enumerate}[label = \textbf{Condition}~\Roman*, leftmargin=*]
  \item 
  $\frac{\partial \varphi}{\partial x_1}(\x) >0,\: \forall \:  \x \in \mathbb{R}^d \: \:\textbf{(Monotonicity condition)}$. \label{condition:a}
   \item  $ \underset{x_1 \rightarrow +\infty}{\lim} \varphi(\x)=\underset{x_1 \rightarrow +\infty}{\lim} 
        \varphi(x_1,\x_{-1})=+\infty, \: \forall \: \x_{-1} \in \mathbb{R}^{d-1}\: \\
        \text{or} \:\: \frac{\partial^2 \varphi} {\partial x_1^2}(\x) \ge 0,\: \forall \: \x \in \mathbb{R}^{d}  \: \:   \textbf{(Growth condition)}$.\label{condition:b}
\end{enumerate}

Under the above conditions, for any fixed $\bm \omega_{-1}$, we can uniquely determine the one-dimensional discontinuity point $y^{\ast}_1$ by solving the equation
\begin{equation*}\label{eq:roots_function}
		\varphi(y^{\ast}_1, \bm \omega_{-1}) = 0.
	\end{equation*}
We employ the Newton iteration method to determine the approximated discontinuity location $\bar{y}^\ast_1$.

The second step of the numerical smoothing framework which is presented is to perform numerical preintegration based on the approximated discontinuity point, which is formulated as follows:
	\begin{equation*}
	    \mathbb{E}[g(\varphi(\bm \omega)] = \mathbb{E}[h(\bm \omega_{-1})]\approx	\mathbb{E}[\hat h(\bm \omega_{-1})] 
		\approx \mathbb{E}[\bar h(\bm \omega_{-1})],
	\end{equation*}

where
	\begin{align}\label{eq:smooth_function_after_pre_integration}
		\hat h_m(\bm \omega_{-1}) &=\int_{\mathbb{R}} \hat g_m(\omega_1, \bm \omega_{-1}) \rho_{1}(\omega_1) d\omega_1 \nonumber\\
		&= \int_{-\infty}^{y^\ast_1} \hat g_m(\omega_1, \bm \omega_{-1}) \rho_{1}(\omega_1) d\omega_1 + \int_{y_1^\ast}^{+\infty} \hat g_m(\omega_1, \bm \omega_{-1}) \rho_{1}(\omega_1) d\omega_1,
	\end{align}
and $\bar{h}$ is the  approximation of $\hat h$  obtained using the Newton iteration and  a two-sided Laguerre quadrature rule, which  is expressed as follows:
	\begin{equation}\label{eq:expression_h_bar}
		\bar h_m(\bm \omega_{-1}) := \sum_{k=0}^{M_{\text{Lag}}} \eta_k \; \hat g_m\left( \zeta_k\left(\bar{y}^{\ast}_1\right),
		\bm \omega_{-1}   \right),
	\end{equation}
where $M_{\text{Lag}}$ represents the number of Laguerre quadrature points  $\zeta_k \in \R$ with $\zeta_0 = \bar{y}^{\ast}_1$ and corresponding weights $\eta_k$.
The quadrature points $\zeta_k$ must be systematically selected based on the approximated discontinuity point $\bar{y}^{\ast}_1$ to ensure numerical accuracy.

Equations~\eqref{eq:smooth_function_after_pre_integration} and \eqref{eq:expression_h_bar} can be easily extended to the case  in which  finitely  many discontinuities exist. We refer to Remark~2.5 presented in \cite{bayer2023numerical} for this extension.

\section{MLMC combined with preintegration}\label{sec:MLMC combined with preintegration}
MLMC method is a powerful variance reduction technique that decomposes the target expectation into a telescoping sum of level differences, significantly reducing computational cost. In this section, we first review the fundamental idea of MLMC, then construct the MLMC estimator with preintegration technique proposed in the previous section. We recall that $ h(\bm \omega_{-1})=\int_{\mathbb{R}} g(\omega_1, \bm \omega_{-1})\rho_1(\omega_1) d\omega_1 $. Instead of dealing with $h$ directly, we consider a sequence of random variables $h_{m_0},h_{m_1},\dots$ with increasing approximation accuracy to $h$ but with increasing cost per sample. The approximation accuracy improves with the increase of inner sample size, while the computational cost per level also rises accordingly. The $\ell$th level approximation of $h$ is $\hat h_{m_{\ell}}(\bm \omega_{-1}) = \int_{\mathbb{R}} \hat g_{m_{\ell}}(\omega_1, \bm \omega_{-1})\rho_1(\omega_1) d\omega_1$, where $m_\ell=2^{\ell+\ell_0}$ in our setting and $\ell_0\ge 0$. We simplify the  subscript notation $m_{\ell}$ to $\ell$, like $h_{m_{\ell}}$ to $h_\ell$. Moreover,  $M_{\text{Lag},\ell}$ and  $\text{TOL}_{\text{Newton},\ell}$  denote the number of Laguerre quadrature points  and  the tolerance of the Newton method at the level $\ell$, respectively. 

Due to the linearity of expectation, we have the following telescoping representation
\begin{equation*}
\mathbb{E}[\hat h_L] = \mathbb{E}[\hat h_0] + \sum_{\ell=1}^L \mathbb{E}[\hat h_{\ell}-\hat h_{\ell-1}].
\end{equation*}
Let $\hat Y_\ell=\hat h_{\ell}-\hat h_{\ell-1}$ for $\ell \ge 1$ and $\hat Y_0=\hat h_0$, we further have
\begin{equation}\label{eq:sumeq}
\mathbb{E}[\hat h_L] = \sum_{\ell=0}^L \mathbb{E}[\hat Y_\ell].
\end{equation}
The MLMC method uses the above equality and estimates each term on the right hand side of \cref{eq:sumeq} independently. This independent estimation strategy allows us to allocate samples optimally across different levels. The resulting MLMC estimator is given by
\begin{equation}\label{eq:hatQ}
    \hat Q = \sum_{\ell=0}^L \hat Q_\ell
\end{equation}
with
$$\hat Q_\ell = \frac{1}{N_\ell}\sum_{i=1}^{N_\ell}\hat Y_\ell^{(i)},\: 0 \le \ell \le L,$$
where $\hat Y_\ell^{(1)},\dots,\hat Y_\ell^{(N_\ell)}$ are i.i.d. replications of $\hat Y_\ell$ for $\ell=0,\dots, L$. 

Similarly, let $\bar Y_\ell=\bar h_{\ell}-\bar h_{\ell-1}$ for $\ell \ge 1$ and $\bar Y_0=\bar h_0$, we can construct the MLMC estimator with numerical preintegration
\begin{equation}\label{eq:barQ}
		\bar{Q}:= \sum\limits_{\ell=0}^{L} \bar{Q}_{\ell},
\end{equation}
where
\begin{equation*}
 	\bar{Q}_{\ell}:= \frac{1}{N_{\ell}}  \sum_{i=1}^{N_\ell}\bar Y_\ell^{(i)}, \: 0 \le \ell \le L,
\end{equation*}
where $\bar Y_\ell^{(1)},\dots,\bar Y_\ell^{(N_\ell)}$ are i.i.d. replications of $\bar Y_\ell$ for $\ell=0,\dots, L$.

\subsection{Error analysis}
This Section  analyzes  the different error contributions in our approach that combines  the MLMC   estimator  with the numerical preinteration to approximate $\theta = \mathbb{E}[g(\varphi(\bm \omega))]$ with $g,\varphi$ given by \cref{eq:theta}. We obtain  the  following MSE decomposition
\begin{align}\label{eq: error decomposition MLMC}
\mathbb{E}\left( \mathbb{E}[g(\varphi(\bm \omega))]-\bar{Q} \right)^2 &\leq \underset{\text{Error I: bias or weak error}}{3\underbrace{\left(\mathbb{E}[g(\varphi(\bm \omega))]-  \mathbb{E}[\hat g_L(\bm \omega)]\right)^2}}\nonumber\\
&+\underset{\text{Error II:  MLMC  statistical error}} {3\underbrace {\mathbb{E}\left(\mathbb{E}[\hat h_L(\bm \omega_{-1})]-\hat Q\right)^2}}\nonumber\\
&+ \underset{\text{Error III: numerical integration and root-finding error}}{3\underbrace{\mathbb{E}\left(\hat{Q}-\bar Q\right)^2}},
\end{align} 
where $\hat h_L$ is given by \cref{eq:hat_h_m}, and $\hat Q$ and $\bar Q$ are given by \cref{eq:hatQ} and \cref{eq:barQ} respectively.

We first analyze Error I, which corresponds to the weak error/bias of the nested estimator. Due to Theorem \ref{thm:bias},    we obtain
\begin{equation}\label{eq:ErrorI}
\text{Error I}=O( m^{-2}_L).
\end{equation}
 
For simplification, we assume	that  on all levels ($0 \le \ell \le L$)   $M_{Lag,\ell}=M_{Lag,L}$ and $\text{TOL}_{\text{Newton},\ell}=\text{TOL}_{\text{Newton},L}$, and then the Error III can compute by
	\begin{align}\label{eq:ErrorIII}
		\text{Error III}&=\mathbb{E}\left(\hat{Q}-\bar Q\right)^2
        =\mathbb{E}\left(\sum_{\ell=0}^L \frac{1}{N_\ell}\sum_{n_\ell=1}^{N_\ell}\left((\hat h_\ell - \hat h_{\ell-1})-(\bar h_\ell-\bar h_{\ell-1})\right)\right)^2\nonumber\nonumber\nonumber\nonumber\\
        &\leq (L+1) \sum_{\ell=0}^L \frac{1}{N_\ell}\sum_{n_\ell=1}^{N_\ell}\mathbb{E}\left((\hat h_\ell - \hat h_{\ell-1})-(\bar h_\ell-\bar h_{\ell-1})\right)^2\nonumber\nonumber\nonumber\\
        &\leq 2(L+1) \sum_{\ell=0}^L \frac{1}{N_\ell}\sum_{n_\ell=1}^{N_\ell}\left(\mathbb{E}(\hat h_\ell - \bar h_{\ell})^2+\mathbb{E}(\hat h_{\ell-1}-\bar h_{\ell-1})^2\right)\nonumber\nonumber\\
		&= O(M_{\text{Lag},L}^{-s})+ O(\text{TOL}^2_{\text{Newton},L}),
	\end{align}
where   $s>0$ is related to the degree of regularity of the integrand, $G$, w.r.t.~$y_1$, and the error bound in \cref{eq:ErrorIII} see \cite{bayer2023numerical, bayer2024multilevel} for details.

Error II represents the statistical error of the MLMC estimator, which is determined by the sample variance and the number of samples at each level. From the standard multilevel error analysis,  we obtain
	\begin{equation}\label{eq:ErrorII}
		\text{Error II} = \sum_{\ell=0}^L {(N_{\ell})}^{-1} \hat V_{\ell},
	\end{equation}
where
\begin{equation*}
	\hat V_{0}:=\text{Var}\left[\hat Y_{0}\right], \:  \hat V_{\ell}:=\text{Var}\left[\hat Y_{\ell}\right], \: 1 \le \ell \le L,
\end{equation*}
and  $C_{\ell}$ is the cost per sample per level, given by
\begin{equation}\label{eq:C_l}
C_{\ell} \propto m_{\ell} (M_{\text{Lag},\ell}+N_{\text{iter},\ell})\propto m_{\ell} \left(M_{\text{Lag},\ell}+\log \left(\text{TOL}_{\text{Newton},\ell}^{-1}\right)\right), \: 0 \le \ell \le L,
\end{equation}
where $N_{\text{iter},\ell}$ is the number of the Newton iterations at level $\ell$. Under some mild conditions and using Taylor expansion, we can  show that Newton iteration has a second order convergence and conclude that $N_{\text{iter},\ell} \propto \log \left(\text{TOL}_{\text{Newton},\ell}^{-1}\right)$ \cite{bayer2024multilevel}.
The variance sequence $\{\hat V_\ell\}$ is the core parameter for MLMC sample allocation, and the next section establishes its decay rate, Theorem \ref{thm:var}  derive  estimates of the variances  $\{\hat V_{\ell}\}_{\ell=0}^{L}$, and show that $\hat V_{\ell}=O(m^{-1}_{\ell})$.
	
Finally, using  \eqref{eq: error decomposition MLMC}, \eqref{eq:ErrorI}, \eqref{eq:ErrorIII} and \eqref{eq:ErrorII}, the total MSE estimate of our approach is
	\begin{align}\label{eq:total_error_estimate MLMC}
		&\mathbb{E}(\mathbb{E}[g(\varphi(\bm \omega))]-\bar{Q})^2\nonumber\\
		&= O(m^{-2}_{L})+O\left(\sum_{\ell=0}^L {(N_{\ell})}^{-1} \hat V_{\ell}\right)+O(M_{\text{Lag},L} ^{-s})+ O(\text{TOL}^2_{\text{Newton},L}).
	\end{align}
\subsection{Strong convergence results for MLMC with preintegration}
In this section, we derive the strong convergence rate of our main work: the MLMC method with the preintegration technique. Before proving our main theorem, we require the following assumptions and definition. We first introduce the following definition for uniform bounded random variables.
\begin{definition}\label{def:O1}
	For sequences of rdvs $F_N$, we write  that
	$F_N = \mathcal{O}(1)$ if there exists a rdv $C$ with finite moments
	of all orders,  such that for all $N$, we have $\abs{F_N} \le C $ a.s.
\end{definition}
We impose the following regularity assumptions on the conditional expectation function and its derivatives.
\begin{assumption}\label{assume:varphi}
Suppose that the function $\varphi(\bm \omega) = \mathbb{E}[X|\bm \omega]$ in \cref{eq:theta} satisfies 
\begin{equation*}
    \varphi(\omega_1,\bm \omega_{-1}) = \mathcal{O}(1), \partial_{\omega_1} \varphi(\omega_1,\bm \omega_{-1}) = \mathcal{O}(1)\ \text{and}\ \partial^2_{\omega_1}\varphi(\omega_1,\bm \omega_{-1}) = \mathcal{O}(1)  .
\end{equation*}
\end{assumption}

\begin{assumption}\label{assume:varphi_1}
Suppose that $\hat\varphi_{\ell}=\hat \varphi_{m_\ell}$ in \cref{eq:hat_g_varphi} satisfies
\begin{equation*}
    \left(\partial_{\omega_1} \hat \varphi_{\ell}(\omega_1,\bm \omega_{-1})\right)^{-1} = \mathcal{O}(1).
\end{equation*}
\end{assumption}
We also assume the score function of the marginal density satisfies an integrability condition.
\begin{assumption}\label{assume:score}
    Suppose that the function $s(y) := \frac{\rho'_1(y)}{\rho_1(y)}$ of the p.d.f. $\rho_1$ exists and satisfies
\begin{equation*}
    \int_{\mathbb{R}} (1+s(y))^p \rho_1(y) dy < \infty.
\end{equation*}
\end{assumption}
The following classical moment inequality for sample averages is essential for our error estimation.
\begin{lemma}\label{lemma:MC_p}
Let $X$ be a random variable with zero mean, and let $\bar{X}_N$ be an average of $N$ i.i.d. samples of $X$. If $\mathbb{E}[|X|^p]$ is finite for $p \geq 2$, there exists a constant $C_p$ depending only on $p$ such that
\[
\mathbb{E}\bigl[|\bar{X}_N|^p\bigr] \leq C_p \frac{\mathbb{E}[|X|^p]}{N^{p/2}}.
\]
\end{lemma}

\begin{proof}
See the proof of \cite{goda_nested_max}.
\end{proof}

Based on the above assumptions and moment inequality, we derive the moment bounds for the level difference error.
\begin{lemma}\label{lemma:e_l}
Suppose that Assumption \ref{assume:varphi} is satisfied and therefore for any $p\geq2$ 
\begin{equation*}
    \mathbb{E}[\varphi^{p}(\bm \omega)] < \infty\  \text{and}\  \mathbb{E}[(\partial_{\omega_1}\varphi)^{p}(\bm \omega)] < \infty .
\end{equation*}
Then we have
	\begin{equation}
			\mathbb{E}[e_{\ell}^{2p}(\bm \omega)]= O(m_{\ell}^{-p}), \label{eq:e_l} 
	\end{equation}
and   
\begin{equation}
	\mathbb{E}[(\partial_{\omega_1}e_{\ell})^{2p}(\bm \omega)]= O(m_{\ell}^{-p}).\label{eq:partial_e_l} 
\end{equation}
\end{lemma}
\begin{proof}
Obviously, under Assumption \ref{assume:varphi}, we can get $ \mathbb{E}[\varphi^{p}(\bm \omega)] < \infty\  \text{and}\  \mathbb{E}[(\partial_{\omega_1}\varphi)^{p}(\bm \omega)] < \infty$ for any $p$.
We define $e_{\ell}(\omega_1 , \bm \omega_{-1}) := \hat{\varphi}_{\ell}(\omega_1 , \bm \omega_{-1})- \hat{\varphi}_{\ell-1}(\omega_1 , \bm \omega_{-1})$, and then we have
\begin{align*}
    \mathbb{E}[e_{\ell}^{2p}(\bm \omega)] &= \mathbb{E}[\hat \varphi_{\ell}(\bm \omega) - \hat \varphi_{\ell - 1}(\bm \omega)]^{2p} \\
    & = \mathbb{E}[(\hat \varphi_{\ell}(\bm \omega) - \varphi(\bm \omega)) - (\hat \varphi_{\ell - 1}(\bm \omega) - \varphi(\bm \omega))]^{2p}\\
    & \leq 2^{2p-1} \mathbb{E}[\hat \varphi_{\ell}(\bm \omega) - \varphi(\bm \omega)]^{2p} + 
     2^{2p-1} \mathbb{E}[\hat \varphi_{\ell-1}(\bm \omega) - \varphi(\bm \omega)]^{2p} .
\end{align*}   
We recall that $\hat \varphi_{\ell}(\bm \omega)$ is an unbiased Monte Carlo estimate of $\varphi(\bm \omega)$ , Hence, as long as $p \geq 2$, it follows from Lemma \ref{lemma:MC_p} that
\begin{align*}
     \mathbb{E}[(\hat \varphi_{\ell}(\bm \omega) - \varphi(\bm \omega))^{2p}|\bm \omega]
     &\leq  C\frac{\mathbb{E}[|X(\bm \omega) - \varphi(\bm \omega)|^{2p}|\bm \omega]}{m_{\ell}^p}\\
     &\leq 2^{2p}C \frac{\mathbb{E}[\varphi^{2p}(\bm \omega)|\bm \omega]}{m_{\ell}^p}.
\end{align*}
Take expectations on both sides of the inequality, we can get
\begin{equation*}
    \mathbb{E}[\hat \varphi_{\ell}(\bm \omega) - \varphi(\bm \omega)]^{2p} = O(m_{\ell}^{-p}).
\end{equation*}
Thus,
\begin{equation*}
    \mathbb{E}[e_{\ell}^{2p}(\bm \omega)]= O(m_{\ell}^{-p}).
\end{equation*}
Similarly, we can obtain that
\begin{equation*}
    \mathbb{E}[(\partial_{\omega_1}e_{\ell})^{2p}(\bm \omega)]= O(m_{\ell}^{-p}).
\end{equation*}
\end{proof}

We now present the theorem of this section, which establishes the optimal variance decay rate for the MLMC estimator with smoothing. To obtain the result of Theorem \ref{thm:var}, we follow the proof strategy of \cite{bayer2024multilevel}.
By incorporating preintegration, we reduce the strong convergence rate to $-1$, compared with $-1/2$ for standard MLMC \cite{giles_2019_mlmc}.

\begin{theorem} \label{thm:var} 	Let  the function $g, \varphi$ as in \eqref{eq:theta} and $\hat h_\ell$ as in \cref{eq:hat_h_m}. Then under  Assumptions \ref{assume:varphi}, \ref{assume:varphi_1} and \ref{assume:score}, we obtain 
	\begin{equation*}
	\hat V_{\ell}=O(m^{-1}_{\ell}) .
	\end{equation*}
\end{theorem}

\begin{proof}\label{proof: Lipschitz of the integrands_revised}
First we observe that  $\hat V_{\ell}  = \text{Var}\left[\hat{h}_{\ell}-\hat{h}_{\ell-1}\right] \le  \mathbb{E}\left[\left(\hat{h}_{\ell}-\hat{h}_{\ell-1}\right)^2\right] $. Therefore, we just need to prove that
$\mathbb{E}\left[\left(\hat{h}_{\ell}-\hat{h}_{\ell-1}\right)^2\right] =O(m^{-1}_{\ell})$.
For $\delta>0$, we define $\tilde{g}_{\delta}$ denotes a $C^\infty$ mollified version of g (i.e., obtained by
convoluting g with a mollifier), then we have

\begin{align}
	&\Delta h^{\delta}_{\ell}(\bm \omega_{-1}):=(\hat{h}^{\delta}_{\ell}-\hat{h}^{\delta}_{\ell-1})( \bm \omega_{-1})\nonumber\\
	&:= \int_{\mathbb{R}} \left(  \tilde{g}_{\delta} (\hat \varphi_{\ell}(\omega_1, \bm \omega_{-1})) - \tilde{g}_{\delta} (\hat \varphi_{\ell-1}(\omega_1, \bm \omega_{-1}))  \right) \rho_1(\omega_1) d\omega_1\label{eq:antith_decomp}\\
	&= \int_{\mathbb{R}} \left[ \int_{0}^{1} \tilde{g}_{\delta}'  \left(\underset{:=z(\theta; \omega_1,\bm \omega_{-1})}{\underbrace{\hat{\varphi}_{\ell-1}(\omega_1 , \bm \omega_{-1})+\theta   e_{\ell}(\omega_1}, \bm \omega_{-1}})\right) d \theta  \right]  e_{\ell}(\omega_1 , \bm \omega_{-1}) \; \rho_1(\omega_1) d\omega_1, \: \theta \in (0,1) \nonumber \\
    &(e_{\ell}(\omega_1 , \bm \omega_{-1}) = \hat{\varphi}_{\ell}(\omega_1 , \bm \omega_{-1})- \hat{\varphi}_{\ell-1}(\omega_1 , \bm \omega_{-1})) \nonumber\\
	&= \int_{\mathbb{R}} \left[ \int_{0}^{1}  \partial_{\omega_1} \tilde{g}_{\delta} (z(\theta;\omega_1 , \bm \omega_{-1}) )  \left(\partial_{\omega_1} z(\theta;\omega_1 , \bm \omega_{-1}) \right)^{-1} d \theta  \right]  e_{\ell}(\omega_1 , \bm \omega_{-1}) \; \rho_1(\omega_1) d\omega_1  \quad  \label{eq:chain rule}\\ 
		&=  \int_{0}^{1} \left[ \int_{\mathbb{R}}   \partial_{\omega_1} \tilde{g}_{\delta} (z(\theta;\omega_1 , \bm \omega_{-1})  )  \left(\partial_{\omega_1} z(\theta;\omega_1 , \bm \omega_{-1})  \right)^{-1}    e_{\ell}(\omega_1 , \bm \omega_{-1})  \; \rho_1(\omega_1) d\omega_1\right]  d\theta  \:(\text{Fubini's theorem})\nonumber\\
		&=-  \int_{0}^{1}  \left[  \int_{\mathbb{R}} \tilde{g}_{\delta} (z(\theta;\omega_1 , \bm \omega_{-1})  )\;  \partial_{\omega_1} \left(\left(\partial_{\omega_1} z(\theta;\omega_1 , \bm \omega_{-1})  \right)^{-1}  e_{\ell}(\omega_1 , \bm \omega_{-1})  \; \rho_1(\omega_1)  \right)  d\omega_1 \right] d \theta \label{eq:boundary}\\ 
		&=  -\int_{0}^{1}  \left[  \int_{\mathbb{R}} \tilde{g}_{\delta}(z(\theta;\omega_1 , \bm \omega_{-1}) )  \left(   e_{\ell}(\omega_1 , \bm \omega_{-1})  \; \partial_{\omega_1} \left(\left(\partial_{\omega_1} z(\theta;\omega_1 , \bm \omega_{-1})  \right)^{-1}  \rho_1(\omega_1) \right) \right.\right.\nonumber\\
        &+ \left.\left. \left(\partial_{\omega_1} z(\theta;\omega_1 , \bm \omega_{-1})  \right)^{-1}    \rho_1(\omega_1)\; \partial_{\omega_1} e_{\ell}(\omega_1 , \bm \omega_{-1})   \right)  d\omega_1 \right] d \theta\nonumber\\
			&=  -\int_{0}^{1}  \left[  \int_{\mathbb{R}} e_{\ell}(\omega_1 , \bm \omega_{-1}) \tilde{g}_{\delta}(z(\theta;\omega_1 , \bm \omega_{-1}) )   \left(   \partial_{\omega_1} \left(\left(\partial_{\omega_1} z(\theta;\omega_1 , \bm \omega_{-1})  \right)^{-1}     \right) + s(\omega_1)  \left(\partial_{\omega_1} z(\theta;\omega_1 , \bm \omega_{-1})  \right)^{-1}\right)   \rho_1(\omega_1)d\omega_1 \right] d \theta\nonumber\\ 
				&- \int_{0}^{1}  \left[  \int_{\mathbb{R}}  \partial_{\omega_1} e_{\ell}(\omega_1 , \bm \omega_{-1})  \tilde{g}_{\delta}(z(\theta;\omega_1 , \bm \omega_{-1}) )  \left(\partial_{\omega_1} z(\theta;\omega_1 , \bm \omega_{-1})  \right)^{-1}     \rho_1(\omega_1) d\omega_1 \right] d \theta,\label{eq:proof_result1}
\end{align}
where $s$ in \cref{eq:proof_result1} is $s(x) = \rho'_1(x)/\rho_1(x)$. The equation \cref{eq:chain rule} uses 
$\partial_{\omega_1} \tilde{g}_{\delta} = \tilde{g}_{\delta}'   \partial_{\omega_1} z$, and the boundary terms of \cref{eq:boundary} vanish due to Lemma \ref{lemma: boundary_condition_error growth}. 

Taking $\delta\rightarrow 0$ and applying  the dominated convergence theorem to   \cref{eq:proof_result1}, we  obtain
\begin{small}
	\begin{align}
		\Delta h_{\ell}(\bm \omega_{-1})&:=(h_{\ell}-h_{\ell-1})(\bm \omega_{-1})\nonumber\\
		&=  \underset{(I)}{\underbrace{{ -\int_{0}^{1}  \left[  \int_{\mathbb{R}} e_{\ell}(\omega_1 , \bm \omega_{-1})  g(z(\theta;\omega_1 , \bm \omega_{-1}) )   \left(   \partial_{\omega_1} \left(\left(\partial_{\omega_1} z(\theta;\omega_1 , \bm \omega_{-1})  \right)^{-1}     \right) + s(\omega_1) \left(\partial_{\omega_1} z(\theta;\omega_1 , \bm \omega_{-1})  \right)^{-1}\right)  \rho_1(\omega_1) d\omega_1 \right] d \theta}}}\nonumber\\
		& \underset{(II)}{\underbrace{{- \int_{0}^{1}  \left[  \int_{\mathbb{R}}  \partial_{\omega_1} e_{\ell}(\omega_1 , \bm \omega_{-1})  g(z(\theta;\omega_1 , \bm \omega_{-1}) )  \left(\partial_{\omega_1} z(\theta;\omega_1 , \bm \omega_{-1})  \right)^{-1}       \rho_1(\omega_1) d\omega_1 \right] d \theta}}}. \label{eq:delta_hl}
	\end{align}
\end{small}

For the   term (I) in   \eqref{eq:delta_hl},  taking expectation  and using H\"older's inequality twice ($p,q, p_1,q_1 \in (1,+\infty)$,  $\frac{1}{p}+\frac{1}{q}=1$ and $\frac{1}{p_1}+\frac{1}{q_1}=1$), result in

	\begin{align}
		\mathbb{E}\left[\left( I\right)^2\right]	
		&\le  \mathbb{E}\left[\left | \left|   g(z(\cdot; \cdot, \bm \omega_{-1})) \left(   \partial_{\omega_1} \left(\left(\partial_{\omega_1} z(\cdot;\cdot,\bm \omega_{-1}) \right)^{-1}     \right) + s(\cdot) \left(\partial_{\omega_1} z(\cdot;\cdot,\bm \omega_{-1}) \right)^{-1}\right)\right|\right|^2_{L_{\rho_1}^q ([0,1] \times \mathbb{R} )}  \right.\nonumber \\
        &\times  \left.\left|\left|   e_{\ell}(\cdot, \bm \omega_{-1})\right|\right|^2_{L_{\rho_1}^p (\mathbb{R})} \right]	\nonumber\\
		&\le  \left(\mathbb{E}\left[\left | \left|   g(z(\cdot; \cdot,\bm \omega_{-1}))  \left(   \partial_{\omega_1} \left(\left(\partial_{\omega_1} z(\cdot;\cdot,\bm \omega_{-1}) \right)^{-1}     \right) + s(\cdot)  \left(\partial_{\omega_1} z(\cdot;\cdot,\bm \omega_{-1}) \right)^{-1}\right)\right|\right|^{2q_1}_{L_{\rho_1}^q ([0,1]  \times \mathbb{R} )}\right] \right)^{1/q_1}\nonumber\\
		&  \times  \left(\mathbb{E}\left[\left|\left|    e_{\ell}(\cdot, \bm \omega_{-1})\right|\right|^{2p_1}_{L_{\rho_1}^p (\mathbb{R})} \right]\right)^{1/p_1}.\label{eq:E_I}
	\end{align}
The first term in  the right-hand side  of \eqref{eq:E_I} is bounded, because we observe that  

\begin{align}
(\partial_{\omega_1} z(\theta;\omega_1 , \bm \omega_{-1}))^{-1}&=   \left(\partial_{\omega_1} \hat \varphi_{\ell-1}(\omega_1, \bm \omega_{-1})\right)^{-1}  \left((1- \theta) + \theta \frac{\partial_{\omega_1} \hat \varphi_{\ell}(\omega_1, \bm \omega_{-1})}{\partial_{\omega_1} \hat \varphi_{\ell-1}(\omega_1, \bm \omega_{-1})}\right)^{-1},\label{eq:partial_z_1}\\
\partial_{\omega_1} \left((\partial_{\omega_1} z(\theta;\omega_1 , \bm \omega_{-1}))^{-1}\right)&=   -\partial^2_y z(\theta;y,B_{\ell})(\partial_y z(\theta;y,B_{\ell}))^{-2},\nonumber\\
&= -\left((1- \theta) \partial^2_{\omega_1} \hat \varphi_{\ell-1}(\omega_1, \bm \omega_{-1})+ \theta \partial^2_{\omega_1} \hat \varphi_{\ell-1}(\omega_1, \bm \omega_{-1})\right) \left(\partial_{\omega_1} \hat \varphi_{\ell-1}(\omega_1, \bm \omega_{-1})\right)^{-2} \nonumber\\
&\times\left((1- \theta) + \theta \frac{\partial_{\omega_1} \hat \varphi_{\ell}(\omega_1, \bm \omega_{-1})}{\partial_{\omega_1} \hat \varphi_{\ell-1}(\omega_1, \bm \omega_{-1})}\right)^{-2}.\nonumber
\end{align}

 Using Assumption \ref{assume:varphi_1}, we obtain that $\left(\partial_{\omega_1} \hat \varphi_{\ell-1}(\omega_1,\bm \omega_{-1})\right)^{-1} $ and $\left(\partial_{\omega_1} \hat \varphi_{\ell-1}(\omega_1,\bm \omega_{-1})\right)^{-2}$ are bounded in moments, i.e.,  $\mathcal{O}(1)$ in the sense of Definition \ref{def:O1}. Moreover, using Assumption \ref{assume:varphi}, we obtain that  $ \partial^2_{\omega_1}  \hat \varphi_{\ell-1}(\omega_1,\bm \omega_{-1})$ and  $ \partial^2_{\omega_1}  \hat \varphi_{\ell}(\omega_1,\bm \omega_{-1})$ are bounded in moments.  These results with Assumption \ref{assume:score} imply  
 \begin{equation}\label{eq:E_I_one}
     \left(\mathbb{E}\left[\left | \left|   g(z(\cdot; \cdot, \bm \omega_{-1})) \left(   \partial_{\omega_1} \left(\left(\partial_{\omega_1} z(\cdot; \cdot, \bm \omega_{-1}) \right)^{-1}     \right) +s(\cdot)  \left(\partial_{\omega_1} z(\cdot; \cdot, \bm \omega_{-1}) \right)^{-1}\right) \right|\right|^{2q_1}_{L_{\rho_1}^q (  [0,1]  \times \mathbb{R})}\right] \right)^{1/q_1}  <\infty.
 \end{equation}

Choosing $p$ and $p_1$ such that $\frac{2p_1}{p}\le1$, and applying Jensen's inequality for  the second term in  the right-hand side  of \eqref{eq:E_I}, then using \cref{eq:e_l} in Lemma \ref{lemma:e_l}, we obtain

\begin{align}
	\left(\mathbb{E}\left[\left|\left|    e_{\ell}(\cdot, \bm \omega_{-1})\right|\right|^{2p_1}_{L_{\rho_1}^p (\mathbb{R})} \right]\right)^{1/p_1} &= 	\left(\mathbb{E}\left[\left(\int_{\mathbb{R}}  \left| e^p_{\ell}(\omega_1 , \bm \omega_{-1}) \right| \rho_1 d\omega_1\right)^{\frac{2p_1}{p}} \right]\right)^{1/p_1} 	\nonumber\\
	&\le	\left(\mathbb{E}\left[\int_{\mathbb{R}}   \left|e^p_{\ell}(\omega_1 , \bm \omega_{-1}) \right| \rho_1 d\omega_1\right]\right)^{\frac{2}{p}} 	\nonumber\\
	&= O( m^{-1}_{\ell}). \label{eq:E_I_two}
\end{align}

Consequently, combined with \cref{eq:E_I_one} and \cref{eq:E_I_two} we conclude that  $\mathbb{E}\left[\left( I\right)^2\right]=O( m^{-1}_{\ell})$.

For the   term (II) in   \eqref{eq:delta_hl},  taking expectation  and using H\"older's  inequality twice ($p,q, p_1,q_1 \in (1,+\infty)$,  $\frac{1}{p}+\frac{1}{q}=1$ and $\frac{1}{p_1}+\frac{1}{q_1}=1$), result in

	\begin{align}
		\mathbb{E}\left[\left( II\right)^2\right]	
		&\le  \mathbb{E}\left[\left | \left|   g(z(\cdot; \cdot, \bm \omega_{-1})) \left(\partial_{\omega_1} z(\cdot; \cdot, \bm \omega_{-1}) \right)^{-1}      \right|\right|^2_{L_{\rho_1}^q ( [0,1] \times  \mathbb{R})}   \times \left|\left|   \partial_{\omega_1} e_{\ell}( \cdot, \bm \omega_{-1})\right|\right|^2_{L_{\rho_1}^p (\mathbb{R})} \right]	\nonumber\\
		&\le  \left(\mathbb{E}\left[\left | \left|   g(z(\cdot; \cdot, \bm \omega_{-1})) \left(\partial_{\omega_1} z(\cdot; \cdot, \bm \omega_{-1}) \right)^{-1}\right|\right|^{2q_1}_{L_{\rho_1}^q ([0,1] \times \mathbb{R})}\right] \right)^{1/q_1} \times  \left(\mathbb{E}\left[\left|\left|    \partial_{\omega_1} e_{\ell}(\cdot, \bm \omega_{-1})\right|\right|^{2p_1}_{L_{\rho_1}^p (\mathbb{R})} \right]\right)^{1/p_1}.\label{eq:E_II}
	\end{align}

Similarly to \eqref{eq:E_I_two} and using \cref{eq:partial_e_l} in Lemma \ref{lemma:e_l}, we obtain that  $ \left(\mathbb{E}\left[\left|\left|  \partial_y   e_{\ell}(\cdot,\bm \omega_{-1})\right|\right|^{2p_1}_{L_{\rho_1}^p (\mathbb{R})} \right]\right)^{1/p_1}= O( m^{-1}_{\ell})$.  Moreover, similar to the prove of \cref{eq:E_I_two}, we get the first term in  the right-hand side  of \eqref{eq:E_II} to be bounded. This concludes that $\mathbb{E}\left[\left( II\right)^2\right]=O( m^{-1}_{\ell})$, 
and consequently finishes  the  proof. 
 \end{proof}
\begin{remark}
    We can reduce the variance by using antithetic sampling \cite{bujok2015multilevel, goda_nested_max, giles_2019_mlmc}, a common technique in the MLMC method. However, when the outer function is discontinuous, such as the indicator function in \cref{eq:theta}, antithetic sampling usually yields little improvement. Antithetic sampling generally achieves satisfactory variance reduction only for continuous functions, which aligns well with our proposed smoothed MLMC method. We construct the following antithetic form by replacing $\hat Y_\ell = \hat h_\ell - \hat h_{\ell-1}$ for coupling the consecutive levels
\begin{equation}\label{eq:antith_estimator}
    \hat Y_\ell = \hat h_\ell - \frac{1}{2}(\hat h^{(1)}_{\ell-1}+\hat h^{(2)}_{\ell-1})    
\end{equation}
where
\begin{equation*}
     \hat h^{(i)}_{\ell-1}(\bm \omega_{-1})=\int_{\mathbb{R}} \hat g^{(i)}_{\ell-1}(\omega_1, \bm \omega_{-1})\rho_1(\omega_1) d\omega_1,\hat g^{(i)}_{\ell-1}(\y) = \mathbb{I} \{\hat \varphi^{(i)}_{\ell-1}(\y) > c\}
\end{equation*}
and
\begin{equation*}
    \hat{\varphi}_{{\ell-1}}^{(i)}(\y)=\frac 1 {m_{\ell-1}}\sum_{j=1+(i-1)m_{\ell-1}}^{im_{\ell-1}} X_j(\bm \y),\ i=1,2.
\end{equation*}
We can prove that the strong convergence rate of antithetic sampling remains $O(m_\ell^{-1})$
, which can be shown by decomposing \cref{eq:antith_decomp} into two parts. Although antithetic sampling does not alter our theoretical convergence order, it exhibits favorable performance in the numerical experiments presented in Section \ref{sec:Numerical experiments}.
\end{remark}
\subsection{Complexity analysis}

To achieve a root mean square error (RMSE) for certain tolerance $\text{TOL}$ with minimal computational cost, we formulate the following constrained optimization problem for the MLMC parameters, one needs to solve  \eqref{eq:opt_MLMC_work} using  \eqref{eq:total_error_estimate MLMC}
	\begin{align}\label{eq:opt_MLMC_work}
		\begin{cases} 
			\underset{ \left( L,\{N_{\ell}\}_{\ell=0}^L,\{M_{\text{Lag},\ell}\}_{\ell=0}^L , \{\text{TOL}_{\text{Newton},\ell}\}_{\ell=0}^L \right)}{\operatorname{min}} \: \mathcal{C} \left( L,\{N_{\ell}\}_{\ell=0}^L,\{M_{\text{Lag},\ell}\}_{\ell=0}^L , \{\text{TOL}_{\text{Newton},\ell}\}_{\ell=0}^L \right)  \\
			s.t. \:   \mathcal{\mathbb{E}}[(\bar Q-\theta)^2]=\text{TOL}^2
		\end{cases}
	\end{align}
where $ \mathcal{C} \left( L,\{N_{\ell}\}_{\ell=0}^L,\{M_{\text{Lag},\ell}\}_{\ell=0}^L , \{\text{TOL}_{\text{Newton},\ell}\}_{\ell=0}^L \right)$ is the total computational complexity of our method.
From  \eqref{eq:C_l}, we obtain the total computational complexity of our approach as follows
\begin{align}\label{eq:MLMC_work}
\mathcal{C} &\left( L,\{N_{\ell}\}_{\ell=0}^L,\{M_{\text{Lag},\ell}\}_{\ell=0}^L , \{\text{TOL}_{\text{Newton},\ell}\}_{\ell=0}^L \right)  \propto   \sum_{\ell=0}^L N_{\ell} C_{\ell} \nonumber\\
& \propto   \sum_{\ell=0}^L N_\ell m_\ell  \left(M_{\text{Lag},\ell}+ \log\left(\text{TOL}_{\text{Newton},\ell}^{-1}\right)  \right) .
\end{align}
In this work, we do not solve \eqref{eq:opt_MLMC_work}; however, we select the different parameters   heuristically. A further investigation of optimizing  \eqref{eq:opt_MLMC_work} is left for a future study.

The following corollary summarizes the computational complexity of the MLMC with preintegration and MLMC with numerical preintegration.
\begin{remark}
Since we have improved the strong convergence rate from $-1/2$ to $-1$ by incorporating the preintegration, our MLMC with preintegration also achieves a sub-optimal computational complexity. Using numerical preintegration introduces additional computational cost, but the constant $s\gg1$ in the $O(\text{TOL}^{-2-2/s} \left(\log(\text{TOL})\right)^2)$  usually makes this extra cost negligible \cite{bayer2024multilevel}.    
\end{remark}
\begin{corollary}\label{corrol: Complexity of MLMC_smoothing}
 Under  the Assumptions of  Theorem \ref{thm:bias} and Theorem \ref{thm:var} ,	for any $0<\text{TOL}<1$, the complexity of MLMC  with numerical preintegration is $O(\text{TOL}^{-2-2/s} \left(\log(\text{TOL})\right)^2)$ . Moreover, the complexity of MLMC  with preintegration is $O(\text{TOL}^{-2} \left(\log(\text{TOL})\right)^2)$.
\end{corollary}

\begin{proof}\label{proof: complexity}
We recall the error decomposition in \cref{eq: error decomposition MLMC}.
			   First, to have Error II to be  $\text{TOL}^2/3$, we  choose 
			 	\begin{equation}\label{eq:N_l}
			 N_{\ell}  \le C\; \text{TOL}^{-2} (L+1) m_{\ell}^{-1}+1,  
			 \end{equation}
			where $C$ is a constant, 		and  for Error I to be  $\text{TOL}^2/3$,   we obtain  
			 \begin{equation*}
			     L+1= O( \log(\text{TOL}^{-1})),
			 \end{equation*}
             and due to $0<\text{TOL}<1$, we have
             \begin{equation*}
                 \sum_{\ell=0}^L m_{\ell} = O(\text{TOL}^{-2}).
             \end{equation*}
		
			 Moreover, to bound Error III by  $\text{TOL}^2/3$, and using \cref{eq:ErrorIII}, we obtain 
			 \begin{align*}
			 M_{\text{Lag},L}&= O(\text{TOL}^{-2/s}),\\
			 N_{\text{iter},L}&= \log\left(\text{TOL}_{\text{Newton},L}^{-1}\right) =O(\log\left(\text{TOL}^{-1}\right)).
			 	\end{align*}
		For simplification, we assume	that  on all levels ($0 \le \ell \le L$)   $M_{Lag,\ell}=M_{Lag,L}$ and $\text{TOL}_{\text{Newton},\ell}=\text{TOL}_{\text{Newton},L}$. Then, Using \cref{eq:MLMC_work} and \cref{eq:N_l}, we have the computational complexity of our MLMC estimator with numerical smoothing is  
		
				\begin{align}
				\sum_{\ell=0}^L N_{\ell} C_{\ell}  &\propto  	\sum_{\ell=0}^L N_{\ell} m_{\ell}\left(M_{\text{Lag},\ell}+\log \left(\text{TOL}_{\text{Newton},\ell}^{-1}\right)\right) \nonumber\\
				& \le  C \: \text{TOL}^{-2} (L+1)^2  \left(M_{\text{Lag},L}+\log \left(\text{TOL}_{\text{Newton,L}}^{-1}\right)\right) \nonumber\\
                &+   \left(M_{\text{Lag},L}+\log \left(\text{TOL}_{\text{Newton,L}}^{-1}\right)\right)  \sum_{\ell=0}^L  m_{\ell}\nonumber\\
				& =    O(\text{TOL}^{-2-(2/s)} \left(\log(\text{TOL})\right)^2).\nonumber
			\end{align}

	\end{proof}

\subsection{Robustness analysis}
Due to the discontinuity of the indicator function \(g\) in \cref{eq:theta}, the standard MLMC method exhibits a high kurtosis phenomenon, which adversely affects its numerical stability and robustness.
 We define $Y_{\ell}:= \hat g_{m_\ell}-\hat g_{m_{\ell-1}}$. The standard deviation of the sample variance for  $Y_{\ell}$ is given by 
	\begin{equation*}
		\sigma_\ell :=\frac{\text{Var}[Y_{\ell}]}{\sqrt{N_{\ell}}}  \sqrt{(\kappa_{\ell}-1)+\frac{2}{N_{\ell}-1}},
	\end{equation*}
	where $\kappa_{\ell}$ is the kurtosis  at level $\ell$, given by
	\begin{equation*}
	\kappa_{\ell} :=	\frac{\mathbb{E}{\left( Y_{\ell}-\mathbb{E}{Y_{\ell}}\right)^4}}{\left(\text{Var}\left[ Y_{\ell}\right]\right)^2}.
	\end{equation*}
 High kurtosis makes it challenging to estimate $V_\ell$ accurately in deeper levels, since there are less samples in deeper levels under the MLMC scheme \cite{ben2020importance}. It should be noted that the high-kurtosis phenomenon is due to the structure of the indicator function. Therefore, the smoothed MLMC method we proposed earlier can precisely solve this problem.
 We define the kurtosis of smoothed MLMC method
    \begin{equation}\label{eq:kurt}
	\hat\kappa_{\ell}:=	\frac{\mathbb{E}{\left(\hat Y_{\ell}-\mathbb{E}{\hat Y_{\ell}}\right)^4}}{\left(\text{Var}\left[\hat Y_{\ell}\right]\right)^2}.
	\end{equation}

\begin{corollary}\label{corollary:kurt}
Let $\hat\kappa_\ell$ in \cref{eq:kurt} be the kurtosis of $\hat Y_\ell$, assume that $\hat V_\ell^{-1} = O(m_\ell)$.
Then under the same assumption of Theorem \ref{thm:var}, we obtain
\begin{equation*}
    \hat \kappa_\ell = O(1).
\end{equation*}
\end{corollary}
\begin{proof}
    Similar to the proof of Theorem \ref{thm:var}, we can obtain $\mathbb{E}{\left(\hat Y_{\ell}-\mathbb{E}{\hat Y_{\ell}}\right)^4} = O(m_\ell^{-2})$. Then, combined with $\hat V_\ell^{-1} = O(m_\ell)$, we get $\hat \kappa_\ell = O(1)$.
\end{proof}
\begin{remark}
In our previous analysis, we only considered the upper bound of \(\hat V_\ell\), for which we conducted theoretical analysis and derived the upper bound \(O(m_\ell^{-1})\).
However, to obtain the kurtosis of \(\hat Y_\ell\), we need a lower bound of \(\hat V_\ell\), which is rather difficult in theoretical analysis.
In numerical experiments for practical problems, we observed that the convergence rate of \(\hat V_\ell\) is roughly on par with \(-1\) and does not converge faster.
Therefore, in Corollary \ref{corollary:kurt}, we assume that \(\hat V_\ell\) has an exact strong convergence rate of \(-1\), i.e., $\hat V_\ell^{-1} = O(m_\ell)$.
\end{remark}

\section{Numerical experiments}\label{sec:Numerical experiments}
In this section, a series of numerical experiments are carried out to investigate
the performance of the proposed method. In the numerical analysis, we focus on a portfolio composed of European options that are written on $d$ underlying stocks, with the price dynamics of these stocks governed by the Black-Scholes framework. For simplicity, we make the assumption that all stock returns are identical and denoted by $\mu$, whereas the risk-free interest rate is represented by $\mu_0$. The price dynamics of the stock $\bm S_t=(S^1_t,\dots,S^d_t)$ is described by the following stochastic differential equation:
\begin{equation}\label{eq:sde_stock}
\frac{d S^i_t}{S_t^i}=\mu' d t+\sum_{j=1}^d\sigma_{ij}d B^i_t,\ i=1,\dots,d,
\end{equation}
where $\mu'=\mu$ under the real-world probability measure, and $\mu'=\mu_0$ under the risk-neutral probability measure. Here, $\bm B_t=(B^1_t,\dots,B^d_t)$ is a $d$-dimensional standard Brownian motion which represents $d$ risk factors in the model and $\sigma_{ij}$ are the corresponding
volatility effects. The solution of \cref{eq:sde_stock} is
analytically available
\begin{equation*}
S^i_t=S^i_0\exp\{(\mu'-\frac{1}{2}\sum_{j=1}^d\sigma^2_{ij})t+\sum_{j=1}^d\sigma_{ij}B_t^j\},\ i=1,\dots,d,
\end{equation*}
where $\bm S_0=(S^1_0,\dots,S^d_0)$ are the initial prices of stocks.

We assume that the maturities of all the European options in the portfolio are the same, denoted by $T$. We want to measure the portfolio risk at a future time $\tau$ ($\tau<T$). In the simulation, we first simulate the random variable $\bm \omega=\bm S_\tau=(S^1_\tau,\dots,S^d_\tau)$ under real-world probability measure, which denotes the prices of stocks at the risk horizon $\tau$ as the outer sample. And then simulate $\bm S_T=(S^1_T,\dots,S^d_T)$ under the risk-neutral probability measure which is the prices of stocks at maturity $T$ given $\bm \omega$ as the inner samples. Let $V_0=\sum_{i=1}^d v^i_0$ denote the initial value of the portfolio, where each $v^i_0$ can be calculated using the Black-Scholes pricing formula \cite{hull2016}. The change in the portfolio's value is 
\begin{equation*}
\varphi(\bm\omega) := V_0-\mathbb{E}[V_T(\bm S_T)|\bm\omega],
\end{equation*}
where $V_T(\bm S_T)$ is the discounted payoff of the portfolio at time $T$ which is a known function of $\bm S_T$.

The target of our numerical experiments is to estimate the probability of loss $\theta=\mathbb{P}(\varphi(\bm\omega)>c)$ for a given threshold $c$.
The number of inner samples is set as $m_\ell=32\times 2^\ell$ in all experiments.

\subsection{Multiple assets}
Now we consider a portfolio consists of $d$ European call options, which was studied in \cite{multiple_assets}. Given the outer sample $\bm\omega=(\omega_1,\dots,\omega_d)=\bm S_\tau=(S^1_\tau,\dots,S^d_\tau)$ which denotes the prices of stocks at the risk horizon $\tau$ under real-world measure
\begin{equation*}
    \omega_i = S^i_\tau =S^i_0\exp\{(\mu-\frac{1}{2}\sum_{j=1}^d\sigma^2_{ij})\tau+\sum_{j=1}^d\sigma_{ij}\sqrt{\tau}Z_j\},\ i=1,\dots,d,
\end{equation*}
where $\bm Z=(Z_1, \dots, Z_d) \sim N(\bm 0,\bm I_d) $ and $\bm K=(K_1,\dots,K_d)$, the strike prices for call options, the portfolio value change is
\begin{equation*}
\varphi(\bm\omega) = V_0-\mathbb{E}[e^{-\mu_0(T-\tau)}\sum_{i=1}^d(S^i_T(\omega_i,\bm W)-K_i)^+|\bm\omega],
\end{equation*}
 the expectation is taken over the random variable
$\bm W=(W^1,\dots,W^d)\sim N(\bm 0,\bm I_d)$, and samples of
$$S^i_T(\omega^i,\bm W)
=\omega^i\exp\{(\mu_0-\frac{1}{2}\sum_{j=1}^d\sigma^2_{ij})(T-\tau)+\sum_{j=1}^d\sigma_{ij}\sqrt{T-\tau}W^j\},\ i=1,\dots,d,$$
are simulated under the risk-neutral measure.

Next, we provide the analytically expression for the preintegration of the loss function with respect to $\omega_1$. 
 For simplicity, we suppose that $\sigma_{1j} = \sigma_{i1} = 0$ for $ 1<i,j\le d$. We rewrite the function $\varphi(\bm \omega)$ as 
\begin{align*}
    \varphi(\bm \omega) &= -\varphi_1(\omega_1) + \varphi_{-1}(\bm \omega_{-1})\\
    &:= -\mathbb{E}[e^{-\mu_0(T-\tau)}(S^1_T(\omega^1,\bm W)-K^1)^+|\omega_1] \\
    &+V_0-\mathbb{E}[e^{-\mu_0(T-\tau)}\sum_{i=1}^d(S^i_T(\omega_i,\bm W)-K_i)^+|\bm\omega_{-1}],
\end{align*}
where
\begin{equation*}
    \omega_1 = S^1_\tau=S^1_0\exp\{(\mu-\frac{1}{2}\sigma^2_{11})\tau+\sigma_{11}\sqrt{\tau}W^1\},
\end{equation*}
and
\begin{equation*}
    \omega_i = S^i_\tau=S^i_0\exp\{(\mu-\frac{1}{2}\sum_{j=2}^d\sigma^2_{ij})\tau+\sum_{j=2}^d\sigma_{ij}\sqrt{\tau}W^j\},\ i=2,\dots,d.
\end{equation*}

According to BS model European call option formula 
\begin{equation*}
    \varphi_1(\omega_1) = \omega_1\Phi(d_1) - K_1e^{-\mu_0(T-\tau)}\Phi(d_2),
\end{equation*}
where
\begin{equation*}
    d_1 =  \frac{log\frac{\omega_1}{K^1}+(\mu_0+\frac{\sigma_{11}^2}{2})(T-\tau)}{\sigma_{11}\sqrt{T-\tau}}, d_2 = d_1 - \sigma_{11}\sqrt{T-\tau} .
\end{equation*}
$\varphi_1$ is a monotonically increasing function, because
\begin{equation*}
    \partial_{\omega_1} \varphi_1(\omega_1) = \Phi(d_1) >0,
\end{equation*}
where $\partial_{\omega_1}\varphi$ is the Delta (Greeks) of $\varphi_1$ and due to monotonicity, there exists a function $\psi_d$ such that $\{\varphi(\bm \omega) > c\} = \{\varphi(\psi_d(\bm\omega_{-1}), \bm \omega_{-1}) > c\} = \{\omega_1 \le \psi_d\}$.
\begin{align*}
    \mathbb{E}[\mathbb{I}\{\varphi(\bm \omega)>c\}|\bm \omega_{-1}] &= \int_{\mathbb{R}} \mathbb{I}\{\varphi_1(\omega_1)\le\varphi_{-1}(\bm \omega_{-1})-c\} \rho_1(\omega_1)d\omega_1 \\
    &= \int_{\mathbb{R}} \mathbb{I}\{\omega_1 \le \psi_d\}\rho_1(\omega_1) d\omega_1\\
    &= \Phi\left(\frac{log\frac{\psi_d}{S_0^1}-(\mu-\frac{\sigma^2_{11}}{2})\tau}{\sigma_{11}\sqrt{\tau}}\right).
\end{align*}
In our numerical experiments, we use the inner simulation to approximate $\varphi_{-1}$ , denoted by $\varphi^{\ell}_{-1}$ in
$\ell$th level, then
\begin{align*}
    \mathbb{E}[\mathbb{I}\{\hat\varphi_\ell(\bm \omega)>c\}|\bm \omega_{-1}] &= \int_{\mathbb{R}} \mathbb{I}\{\varphi_1(\omega_1)\le\hat\varphi^\ell_{-1}(\bm \omega_{-1})-c\} \rho_1(\omega_1)d\omega_1 \\
    &= \int_{\mathbb{R}} \mathbb{I}\{\omega_1 \le \hat\psi^\ell_d\}\rho_1(\omega_1) d\omega_1\\
    &= \Phi\left(\frac{log\frac{\hat\psi^\ell_d}{S_0^1}-(\mu-\frac{\sigma^2_{11}}{2})\tau}{\sigma_{11}\sqrt{\tau}}\right).
\end{align*}
where $\hat \psi_d^\ell$  satisfies $\{\hat\varphi_\ell(\bm \omega) > c\} = \{\hat\varphi_\ell(\hat\psi^\ell_d(\bm\omega_{-1}), \bm \omega_{-1}) > c\} = \{\omega_1 \le \hat\psi^\ell_d\}$.

The parameters in our experiments are set as follows: initial prices of the assets $S^1_0=\dots=S^d_0=100$, $\mu = 8\%$, $\mu_0=5\%$, the strikes $K_1=\dots=K_d=95$, the maturity $T=0.1$, and the risk horizon $\tau =0.02$. Without loss of generality, we let $\Sigma=(\sigma_{ij})$ be a sub-triangular matrix satisfying $C=\Sigma\Sigma^T$ which corresponds with Cholesky decomposition of $C$, where $C_{ij}=0.3\cdot0.98^{|i-j|}$ for $ 1<i,j\le d$ and $C_{11} = 0.3 $.

We use the following algorithm \cite{giles_2015_mlmc, goda_nested_log} to compute the number of outer simulation samples $N_\ell$.

\begin{algorithm}[ht]
\caption{Let $\omega \in (0,1)$ be a user-specified parameter. For a target root-mean-square accuracy $\varepsilon$, start with $L=L_0$ and give an initial number of samples $N_*$ for all the levels $\ell=0,\dots,L$. Until extra samples need to be evaluated, repeat the following:}
\begin{algorithmic}[1]
\STATE Evaluate extra samples on each level.
\STATE Compute (or update) the empirical variances $V'_\ell$ for $\ell=0,\dots,L$.
\STATE Define optimal $N_\ell$ for $\ell=0,\dots,L$ according to
\[
N_\ell = \left\lceil (1-\omega)^{-1}\varepsilon^{-2}\sqrt{\frac{V'_\ell}{C_\ell}} \sum_{\ell=0}^L \sqrt{V'_\ell C_\ell} \right\rceil.
\]
\STATE Test for the bias convergence $|\mathbb{E}[Z_L]|/(2^\alpha - 1) \leq \sqrt{\omega}\varepsilon$, where we use the empirical estimates for $\mathbb{E}[Z_\ell]$ and $\alpha$.
\IF{the bias is not converged}
    \STATE Let $L = L + 1$ and give an initial number of samples $N_L$.
\ENDIF
\end{algorithmic}
\end{algorithm}

In this algorithm, the optimal allocation of \(N_\ell\) given in Item 3 is derived by minimizing
the total cost \(\sum_{\ell=0}^L N_\ell C_\ell\) for a fixed variance \(\sum_{\ell=0}^L V_\ell/N_\ell = (1-\omega)\varepsilon^2\). The bias convergence rate $\alpha$ in Item 4 is $1$, due to the theoretical bias in Theorem \ref{thm:bias}.


All numerical experiments are based on the Black-Scholes model for a portfolio of European call options. For all levels $\ell$, the inner sample size is set as $m_{\ell}=32 \times 2^{\ell}$. The outer sample size $N_{\ell}$ is chosen adaptively by Algorithm 1 with parameters $\omega=0.16$, $L_0=2$ and $N_*=200,000$.  We compare the standard Multilevel Monte Carlo (MLMC) and the proposed MLMC with preintegration (smoothed MLMC) in terms of the absolute value of expectation, variance, kurtosis and total computational cost, with the dimension covering $d=4,8,16$.  As we know, using an antithetic estimator in \cref{eq:antith_estimator} does not change the variance convergence rate in the crude MLMC because of the indicator function. It is of interest to test whether antithetic sampling can get more benefit in the smoothed method. Therefore, we also test the strong convergence rate and computational cost of smoothed MLMC with antithetic sampling (denoted by smoothed AMLMC) in the following experiments. 

\begin{figure}
    \centering 
    \includegraphics[width=1.\textwidth]{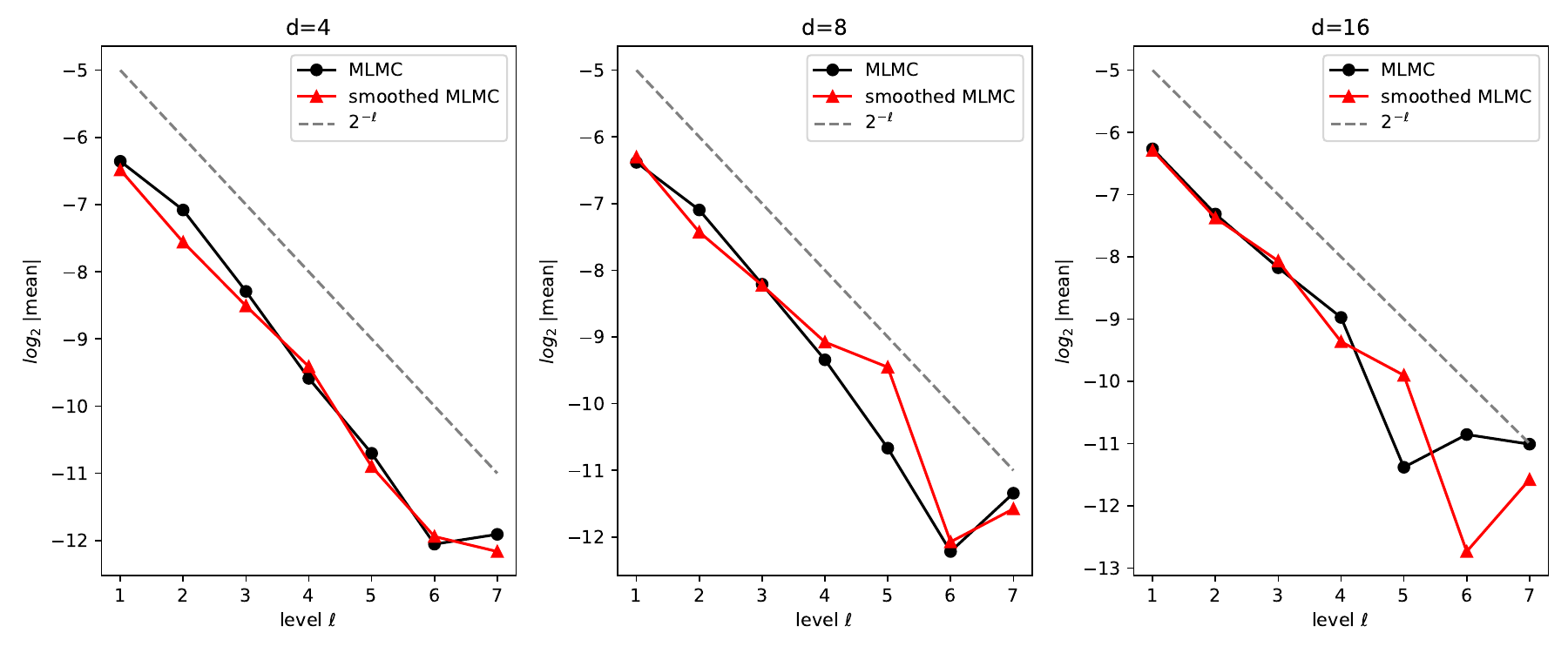}
    \caption{Expectation for dimensions $d=4,8,16$} 
    \label{fg:mean}
\end{figure} 
Figure \ref{fg:mean} shows the weak convergence behavior of MLMC and smoothed MLMC for dimensions $d=4,8,16$. For both methods, $\log_2|mean|$ decreases as the level increases, which means the estimation bias becomes smaller when we use more inner samples.  The convergence rate of the smoothed method is almost the same as the standard MLMC method. This matches our theoretical expectation: preintegration does not change the convergence rate or the bias of the estimator. However, the curve of the smoothed MLMC is much smoother and more stable than standard MLMC. The preintegration method greatly improves the stability of the expectation estimation, even though it does not change the convergence speed.

\begin{figure}
    \centering 
    \includegraphics[width=1.\textwidth]{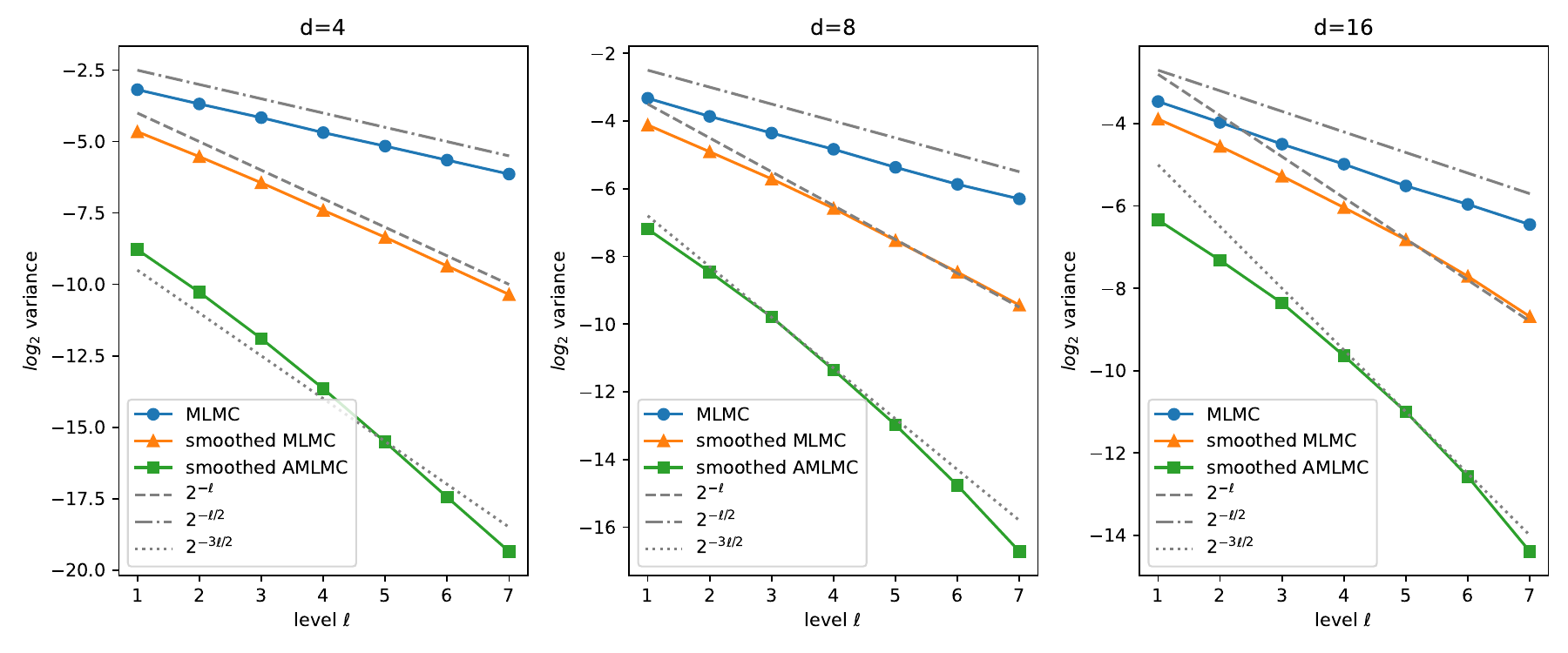}
    \caption{Variance for dimensions $d=4,8,16$} 
    \label{fg:var}
\end{figure}
Figure \ref{fg:var} shows the strong convergence of the three methods for $d=4,8,16$. The variance decay of the standard MLMC is slow and its rate is near $-1/2$, while the variance of the smoothed MLMC decays fast with the increase of the level and its rate is near $-1$, which verifies the conclusion of Theorem \ref{thm:var} that $V_{\ell}=O(m_{\ell}^{-1})$. The AMLMC method achieves a convergence rate even close to $-3/2$. Its variance is reduced by a factor of $10,000$ at the finest level compared with the standard MLMC method. This demonstrates the superiority of our smoothed MLMC method combined with antithetic sampling.

\begin{figure}
    \centering 
    \includegraphics[width=1.\textwidth]{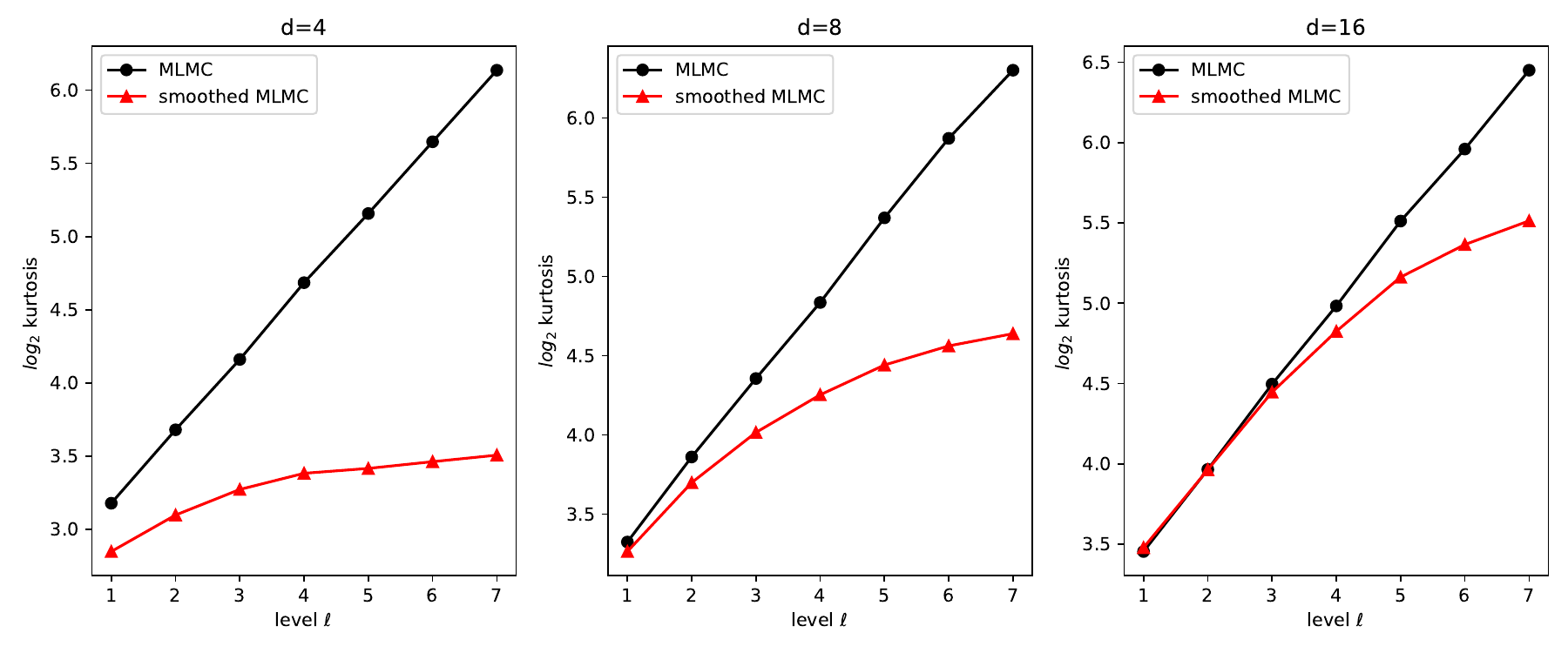}
    \caption{Kurtosis for dimensions $d=4,8,16$} 
    \label{fg:kurt}
\end{figure} 
Figure \ref{fg:kurt} reflects the kurtosis of the two methods. The kurtosis of the standard MLMC increases significantly with the level, which is the high-kurtosis phenomenon caused by the discontinuity of the indicator function. In contrast, the kurtosis of the smoothed MLMC remains at a low level and gradually approaches a certain upper bound, which verifies Corollary \ref{corollary:kurt} that the kurtosis of the proposed method is bounded ($\kappa_{\ell}=O(1)$). The preintegration technique effectively eliminates the discontinuity of the indicator function and improves the robustness of the estimation.

\begin{figure}
    \centering 
    \includegraphics[width=1.\textwidth]{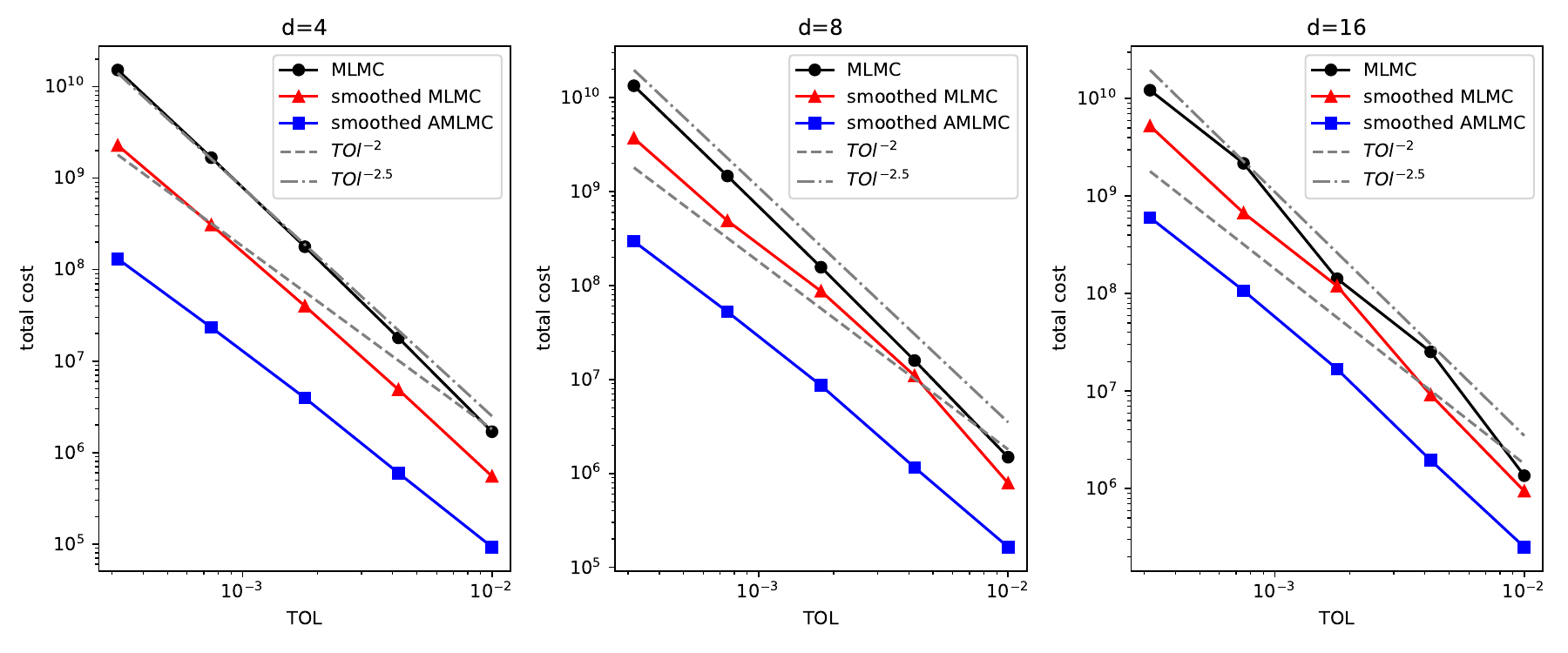}
    \caption{Cost for dimensions $d=4,8,16$} 
    \label{fg:cost}
\end{figure} 
Figure \ref{fg:cost} compares the total computational cost of the three methods for $d=4,8,16$. 
Smooth MLMC achieves a near optimal computational complexity $O(\text{TOL}^{-2})$, compared with the complexity $O(\text{TOL}^{-5/2})$ of standard MLMC. Smoothed AMLMC also attains the optimal complexity $O(\text{TOL}^{-2})$, while reducing the computational burden more effectively, its complexity is 100 times lower than that of the standard MLMC method.

From the variance results above, we observe that the reduction factor of the variance of smoothed MLMC decreases as the dimension increases. This is not a flaw of the smoothed method itself. When we derive the preintegration formula, we use inner simulation to approximate the function $\varphi$. As the dimension increases, the accuracy of the inner simulation decreases, which reduces the accuracy of our preintegration approximation. To fix this accuracy issue of inner simulation in high-dimensional problems, we introduce the quasi-Monte Carlo (QMC) method and combine it with smoothed MLMC.

Next, we incorporate QMC methods within the inner simulation. To fit QMC framework, we assume that given $\bm \omega=\y$, $X$ can be generated via the mapping
\begin{equation*}
X(\y)= \psi(\bm u;\y),
\end{equation*}
for some function $\psi$ and $\bm u \sim U[0,1)^s$ with dimension $s$. As a result, the inner estimator \cref{eq:hat_g_varphi} is replaced by
\begin{equation}\label{eq:qmc_estimator}
\hat \varphi_m(\y) = \frac 1 m\sum_{j=1}^m \psi(\bm u_j;\y),
\end{equation}
and $\varphi(\y) = \int_{[0,1)^s} \psi(\bm u;\y) d \bm u$.
If $\bm u_1,\dots,\bm u_m$ are QMC points (known as low discrepancy points), which are deterministic points chosen from $[0,1)^d$ and are more uniformly distributed than random points, the approximation \cref{eq:qmc_estimator} refers to the QMC method.

The error of the QMC quadrature can be bounded by the Koksma-Hlawka inequality \cite{niederreiter1992}
\begin{equation}
|\varphi(\y)-\hat \varphi_m(\y)|\le V_{\mathrm{HK}}\big(\psi(\cdot;\y)\big)D^*(\bm u_1,\dots,\bm u_m),\label{eq:kh_ineq}
\end{equation}
where $V_{\mathrm{HK}}\big(\psi(\cdot;\y)\big)$ is the variation of the integrand $\psi(\cdot;\y)$ for given $\y$ in the sense of Hardy and Krause which measures the smoothness of $\psi(\cdot;\y)$, and $D^*(\bm u_1,\dots,\bm u_m)$ is the star discrepancy which measures the uniformity of the points set $\{\bm u_1,\dots,\bm u_m\}$. The Koksma-Hlawka inequality \cref{eq:kh_ineq} implies that for functions of finite variation, the convergence rate of QMC approximation is determined by the factor $D^*(\bm u_1,\dots,\bm u_m)$, which is of order $O(m^{-1}(\log m)^s)$ for low discrepancy points.

There are various constructions of QMC point sets in the literature \cite{dick2017}, such as digital nets and lattice rule point sets. For the following numerical experiments, we use $(t,d)$-sequences in base $b\ge 2$.
and also use randomized QMC (RQMC) points. RQMC retains the essential equi-distribution structure, while allowing a statistical error estimation based on independent replications. Moreover, RQMC is able to improve the rate of convergence for smooth integrands, such as the Owen's scrambling technique \cite{owen1995}. For a survey on RQMC, we refer to \cite{l2002}.

\begin{figure}
    \centering 
    \includegraphics[width=1.\textwidth]{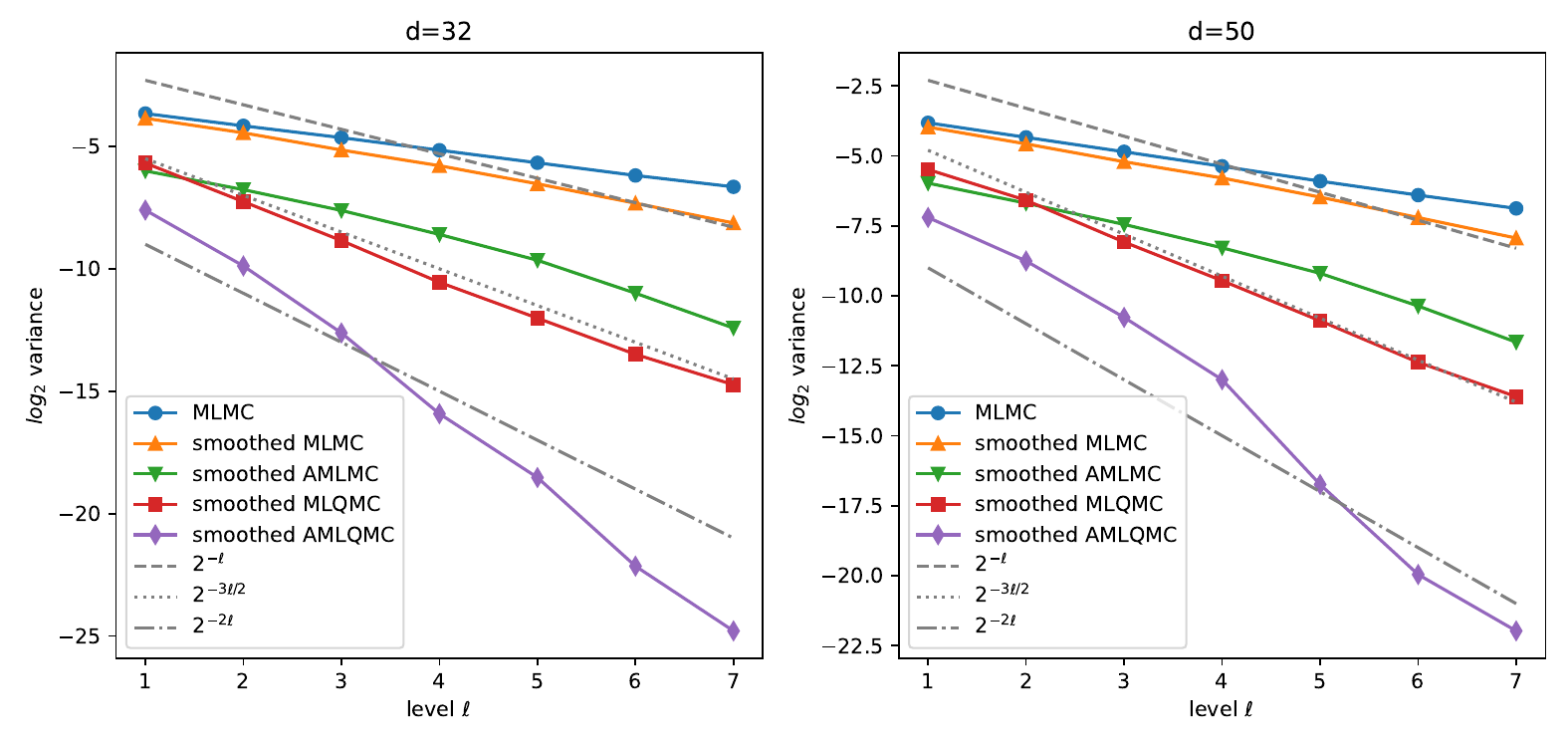}
    \caption{Variance for dimensions $d=32, 50$} 
    \label{fg:var_qmc}
\end{figure} 
\begin{figure}
    \centering 
    \includegraphics[width=0.5\textwidth]{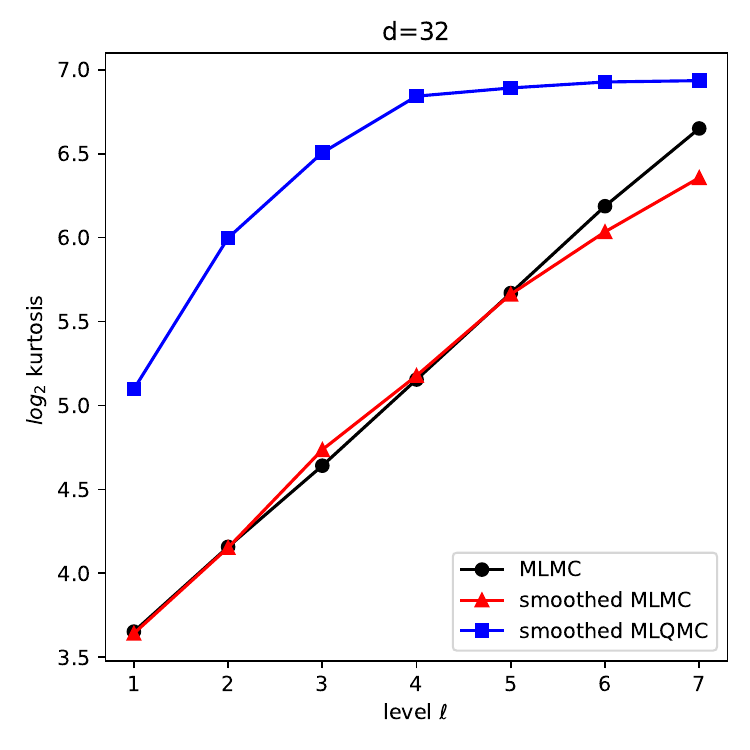}
    \caption{Kurtosis for dimension $d=32$} 
    \label{fg:kurt_qmc}
\end{figure}
Figure \ref{fg:var_qmc} tests the strong convergence performance in high-dimensional cases $d=32,50$. Smoothed MLMC maintains a convergence rate of $-1$ even for high-dimensional problems. Besides, both smoothed MLQMC and smoothed AMLMC exhibit a faster convergence rate approaching $-3/2$, with smoothed MLQMC yielding a smaller variance. The smoothed MLQMC method also demonstrates superior numerical stability, as its kurtosis growth is bounded from above in Figure \ref{fg:kurt_qmc}. Finally, by combining smoothed MLQMC with antithetic sampling (smoothed AMLQMC), its convergence rate even exceeds $-2$, and its variance is reduced by a factor of $1,000,000$ compared to MLMC at the finest level. However, smoothed AMLQMC also suffers from its high strong convergence rate, which leads to an increased kurtosis and degraded numerical stability.

\section{Conclusion}\label{sec:conclusion}

This paper develops a nested MLMC method with preintegration to solve the efficiency and robustness problems of standard nested simulation in risk estimation. The preintegration technique effectively handles the discontinuity of the indicator function in risk measures by integrating out one outer random variable. The theoretical analysis proves that the proposed method has fast variance decay ($O(m_\ell^{-1})$), bounded kurtosis and lower computational complexity compared with the standard MLMC, which means better convergence and robustness.

Numerical experiments on European option portfolios under the Black-Scholes model validate the superiority of the smoothed MLMC with preintegration. It shows more stable weak convergence, faster decaying level variance and bounded kurtosis, which overcomes the high-kurtosis and slow variance decay problems of the standard MLMC. In terms of computational cost, the proposed method requires much less cost to achieve the same accuracy. For high-dimensional problems, combining the proposed method with QMC further improves the convergence performance, solving the accuracy decline of inner simulation in high dimensions.

The proposed method provides a new efficient approach for financial risk estimation, especially for high-dimensional portfolio risk measurement. Future research can focus on optimizing the parameter selection of the MLMC with preintegration, such as the optimal allocation of quadrature points and Newton iteration tolerance across different levels, and extending the method to more complex financial models (e.g., stochastic volatility models) and other risk measures (e.g., expected shortfall). Furthermore, the theoretical analysis combining MLQMC and preintegration has not yet been established, which can also serve as an important direction for future research.
\appendix
\section{Appendix}\label{sec:Appendix}
\subsection{Proof of Lemma \ref{lemma: boundary_condition_error growth}}
\begin{lemma}\label{lemma: boundary_condition_error growth}	Assume that $\varphi$ in \eqref{eq:theta}  satisfy Assumption \ref{assume:varphi}. Then we have

		\begin{equation*}
			\underset{\abs{\omega_1} \rightarrow \infty}{\lim}  	\mathbb{E}\left[  \left(e_{\ell}(\omega_1,\bm \omega_{-1}) \; \rho_1(\omega_1) \int_{0}^{1}  g (z(\theta; \omega_1,\bm \omega_{-1}) )  \left(\partial_{\omega_1} z(\theta; \omega_1,\bm \omega_{-1}) \right)^{-1}   d \theta  \right)^2\right]=0.
		\end{equation*}

\end{lemma}
\begin{proof}
	We have $g(\cdot)$ is bounded and we have also $ \left(\partial_{\omega_1} z(\theta;\omega_1,\bm \omega_{-1}) \right)^{-1}$ is bounded in moments (similar to what we showed for \eqref{eq:partial_z_1} in the proof of Theorem \ref{thm:var}). Consequently, we need  to show that $  	  e^2_{\ell}(\omega_1,\bm \omega_{-1}) \;  =\mathcal{O}(1)$. 
	
	First, we observe that  $  e^2_{\ell}(\omega_1,\bm \omega_{-1}) \le  2\hat \varphi_{\ell}^2(\omega_1,\bm \omega_{-1})+ 2\hat \varphi_{\ell-1}^2(\omega_1,\bm \omega_{-1})$. Therefore, to conclude our target result, we just  need to get a bound on $\ \hat \varphi_{\ell}^2(\omega_1,\bm \omega_{-1}) $.
\begin{align*}
     \hat \varphi_{\ell}^2(\omega_1,\bm \omega_{-1}) &= \left(\frac{1}{m_{\ell}} \sum_{j=1}^{m_\ell} X_j(\omega_1,\bm \omega_{-1})\right)^2\le \frac{1}{m_{\ell}} \sum_{j=1}^{m_\ell}  X_j^2(\omega_1,\bm \omega_{-1})
\end{align*}
Using Assumption \ref{assume:varphi}, we get $ \hat \varphi_{\ell}^2(\omega_1,\bm \omega_{-1}) = \mathcal{O}(1)$.
Finaly, we get 
\begin{equation*}
			\underset{\abs{\omega_1} \rightarrow \infty}{\lim}  	\mathbb{E}\left[  \left(e_{\ell}(\omega_1,\bm \omega_{-1}) \; \rho_1(\omega_1) \int_{0}^{1}  g (z(\theta; \omega_1,\bm \omega_{-1}) )  \left(\partial_{\omega_1} z(\theta; \omega_1,\bm \omega_{-1}) \right)^{-1}   d \theta  \right)^2\right]=0.
		\end{equation*}
	
\end{proof}

\bibliographystyle{plain}
\bibliography{references}
\end{document}